\documentclass[12pt]{amsart}
\usepackage[utf8]{inputenc}

\usepackage[margin=1.3in]{geometry}

\usepackage{times}
\usepackage{amsfonts}
\usepackage{amssymb}
\usepackage{amsthm}

\usepackage{mathtools}
\usepackage{mathrsfs}
\usepackage{caption}
\usepackage{subcaption}
\usepackage{bbm}
\usepackage[export]{adjustbox}
\usepackage{csquotes}

\usepackage{stmaryrd}

\usepackage[all]{xy}

\usepackage{tikz-cd}
\usetikzlibrary{matrix}
\usepackage{graphicx} 
\usepackage{float}

\usepackage{epstopdf}

\usepackage[linktocpage]{hyperref}
\hypersetup{
    colorlinks=true,
    linkcolor=blue,
    citecolor=blue,      
    urlcolor=blue,
}

\usepackage{color}
\definecolor{note}{rgb}{0,0,1}  

\newtheorem{theorem}{Theorem}
\newtheorem{definition}[theorem]{Definition}
\newtheorem{proposition}[theorem]{Proposition}
\newtheorem{lemma}[theorem]{Lemma}

\newtheorem{corollary}[theorem]{Corollary}

\newtheorem{remark}[theorem]{Remark}
\newtheorem{conjecture}[theorem]{Conjecture}

\newtheorem{condition}[theorem]{Condition}
\newtheorem{convention}[theorem]{Convention}

\numberwithin{equation}{section}
\numberwithin{theorem}{section}

\usepackage{enumitem}

\newcommand{\op}{\operatorname}
\newcommand{\s}{\vskip.15in}
\newcommand{\n}{\noindent}
\newcommand{\bdry}{\partial}

\newcommand{\R}{\mathbb{R}}
\newcommand{\C}{\mathbb{C}}
\newcommand{\Z}{\mathbb{Z}}

\newcommand{\be}{\begin{enumerate}}
\newcommand{\ee}{\end{enumerate}}

\usepackage[english]{babel}

\usepackage{hyphenat}

\usepackage[backend=bibtex,style=alphabetic,maxalphanames=4,maxnames=4]{biblatex}

\renewbibmacro{in:}{}

\DeclareDelimFormat[bib,biblist]{nametitledelim}{\addcomma\space}

\DeclareFieldFormat*{title}{\mkbibitalic{#1}\addcomma}
\DeclareFieldFormat*{journaltitle}{#1}
\DeclareFieldFormat*{volume}{\mkbibbold{#1}}
\DeclareFieldFormat{pages}{#1}
\DeclareFieldFormat[misc]{date}{preprint {#1}}
\DeclareFieldFormat{mr}{%
  MR\addcolon\space
  \ifhyperref
    {\href{http://www.ams.org/mathscinet-getitem?mr=MR#1}{\nolinkurl{#1}}}
    {\nolinkurl{#1}}}
    
\AtEveryBibitem{
  \clearfield{url}
  \clearfield{number}
  \clearfield{doi}
  \clearfield{issn}
  \clearfield{isbn}
  \clearfield{eprintclass}
}
\AtEveryBibitem{\ifentrytype{book}{\clearfield{pages}}{}}

\bibliography{hecke}

\usepackage{fancyhdr}
\usepackage{todonotes}

\title{Higher-dimensional Heegaard Floer homology and spectral networks}

\author{Ko Honda}
\address{University of California, Los Angeles, Los Angeles, CA 90095}
\email{honda@math.ucla.edu} \urladdr{http://www.math.ucla.edu/\char126 honda}

\author{Yin Tian}
\address{School of Mathematical Sciences, Beijing Normal University; 
Laboratory of Mathematics and Complex Systems, Ministry of Education, Beijing 100875, China}
\email{yintian@bnu.edu.cn} \urladdr{}

\author{Tianyu Yuan}
\address{School of Mathematical Sciences, Eastern Institute of Technology, Ningbo, Zhejiang, 315200, China}
\email{tyyuan@eitech.edu.cn} \urladdr{}

\date{\today}

\subjclass[2020]{Primary 53D40; Secondary 57K20.}

\begin{document}

\begin{abstract}
    Given a closed surface $C$ and a real exact Lagrangian $\Sigma \subset T^*C$ associated to a spectral curve, we construct a homomorphism $\op{BSk}_\kappa(C)\to\op{Mat}(N^{\kappa},\op{BSk}_\kappa(\Sigma))$ from the braid skein algebra of $C$ to the matrix-valued braid skein algebra of $\Sigma$ using Floer theory and in particular higher-dimensional Heegaard Floer homology (HDHF).  We sketch a proof that this map coincides with a hybrid Floer-Morse approach which counts HDHF-type holomorphic curves coupled with certain Morse gradient graphs --- called fold\-ed Morse trees --- using a variant of the adiabatic limit theorems of Fukaya-Oh and Ekholm, which compares holomorphic curves and Morse flow trees.
\end{abstract}

\maketitle

\tableofcontents

\section{Introduction}

Let $C$ be a closed Riemann surface and $\omega_C$ be the canonical line bundle of $C$. 
Fixing an integer $N \ge 1$, let $\phi_i$ be a meromorphic section of $\omega_C^{\otimes i}$ for $1 \le i \le N$. The data $\phi=(\phi_1,\dots, \phi_N)$ defines a {\em spectral curve} 
\begin{equation}
\label{eq-spectral curve}
    \Sigma_{\phi}=\left\{(z,p_z)~\Bigg|~ z\in C, p_z\in T^*_{\op{hol},z} C,\, P_\phi(z,p_z)\coloneqq p_z^N+\sum_{i=1}^{N}\phi_i(z) p_z^{N-i}=0\right \},
\end{equation}
where $T^*_{\op{hol}}C$ denotes the holomorphic cotangent bundle of $C$.

In the seminal paper \cite{gaiotto2013}, Gaiotto, Moore, and Neitzke introduced the notion of a {\em spectral network} in the study of four-dimensional $\mathcal{N}=2$ supersymmetric gauge theories coupled to surface defects. Roughly speaking, a spectral network is a network of trajectories on $C$ associated to the meromorphic differentials $\phi$ and which obey certain local rules.
A spectral network provides a way of lifting a path on $C$ to that on $\Sigma_\phi$, and hence induces a ``nonabelianization map'' from the moduli space of $\mathfrak{gl}(1)$-flat bundles on $\Sigma_\phi$ to that of $\mathfrak{gl}(N)$-flat bundles on $C$.
This nonabelianization map admits a quantization $\op{Sk}(C,\mathfrak{gl}(N)) \to \op{Sk}(\Sigma_\phi,\mathfrak{gl}(1))$ between the skein algebras. It is called {\em $q$-nonabelianization} in the work of Neitzke-Yan \cite{neitzke2020q}.  

\begin{remark} 
    In this paper we will not define what we mean by a ``spectral network'', referring the reader to \cite{Neitzke2021notes}.  Instead we define a certain class of graphs called {\em folded Morse flow trees} that are ``tangent to'' the spectral network associated to $\pi:\Sigma_\phi\to C$.
\end{remark}

This paper presents another approach to quantizing the nonabelianization map. Instead of skein algebras, we consider {\em braid skein algebras}. The {\em braid skein algebra} of a real surface, studied by Morton-Samuelson~\cite{morton2021dahas}, is generated by $\kappa$-strand braids, with $\kappa$ a positive integer, in a thickening of the surface modulo certain skein relations.  One of the aims of this paper is to use Floer theory to give an algebra map $\mathcal{F}:\op{BSk}_{\kappa}(C) \to \op{Mat}(N^{\kappa},\op{BSk}_\kappa(\Sigma_\phi))$ from the braid skein algebra of $C$ to the matrix-valued braid skein algebra on $\Sigma_\phi$; see Section~\ref{section: skein algebras} for definitions of $\op{BSk}_{\kappa}(C)$ and $\op{Mat}(N^{\kappa},\op{BSk}_\kappa(\Sigma_\phi))$.

The braid skein algebra reduces to the group ring of the fundamental group when $\kappa=1$ and hence our map $\mathcal{F}$ basically recovers the Gaiotto-Moore-Neitzke (GMN) map.  The relationship between the work of GMN and Floer theory was studied by Nho \cite{nho2023family}. A new ingredient of our map is to interpret the skein relation in terms of Floer theory and higher-dimensional Heegaard Floer homology (HDHF) in particular. The relationship between skein theory and holomorphic curves was first studied by Ekholm-Shende~\cite{ekholm2021skeins}, giving a curve counting interpretation of the HOMFLY-PT polynomial. 
Building on this work, we exhibited an isomorphism between the braid skein algebra of a surface and the wrapped HDHF ($A_\infty$-)algebra of cotangent fibers of the surface in \cite{honda2024jems}. We extend the isomorphism to incorporate spectral networks here. 
$$\begin{tikzcd}[row sep=large, column sep = large]
    \mathfrak{gl}(N)\text{-local systems on }C   & \mathfrak{gl}(1)\text{-local systems on }\Sigma_\phi\arrow{l}{\text{GMN/Nho}} \\
    \op{BSk}_\kappa(C) \arrow{u}{\kappa=1} \arrow{r}{\text{this paper}} & \op{Mat}(N^{\kappa},\op{BSk}_\kappa(\Sigma_\phi))\arrow{u}{\kappa=1}
\end{tikzcd}$$

We are interested in the connection between real symplectic geometry and spectral networks via Floer theory. 
A few works have appeared recently along the same lines. 
Casals and Nho study Betti Lagrangians with Legendrian asymptotics in spectral networks \cite{casals2025spectralnetworksbettilagrangians}. 
In recent work, Ekholm, Longhi, Park, and Shende \cite{ekholm2025skeintracescurvecounting} discuss the $q$-nonabeliani\-zation map via open Gromov-Witten theory.

From now on, we assume the following:
\begin{condition}\label{cond: nonsingularity}
    $P_\phi=0$, $\frac{\partial P_\phi}{\partial z}=0$, and $\frac{\partial P_\phi}{\partial p_z}=0$ do not hold simultaneously for $P_\phi$ in  \eqref{eq-spectral curve}.
\end{condition}

Then $\Sigma_\phi$ is smooth by the implicit function theorem. Since Condition~\ref{cond: nonsingularity} is a real codimension-$2$ condition on $\phi$, $\Sigma_\phi$ is smooth for a generic choice of $\phi$.

The projection $\pi:\Sigma_{\phi}\to C$ is an $N$-fold branched cover over $C$, where the ramification points are exactly $\Sigma_{\phi}\cap \{\frac{\partial P_\phi}{\partial p_z}=0\}$. 

\begin{remark}\label{remark: punctures}
    Since $P_\phi$ may admit poles, $\Sigma_\phi$ is actually a punctured Riemann surface, and ``$\pi: \Sigma_\phi\to C$ is a branched cover over $C$'' actually means that $\pi$ extends to a branched cover after putting back the set of punctures $\widetilde P$ to $\Sigma_\phi$.
\end{remark}

We further require:

\begin{condition}
    All the ramification points are simple.
\end{condition} 

At this point we transition from the holomorphic cotangent bundle $(T^*_{\op{hol}}C,\Omega=dz\wedge dp_z)$ to the real cotangent bundle $(T^*C, \omega=\op{Re}\Omega)$ of $C$; see Section~\ref{section: holomorphic symplectic to real symplectic} for more details. Under this transition, the spectral curve $\Sigma_{\phi}$ becomes a real Lagrangian submanifold in $T^*C$, denoted $\Sigma$ for simplicity.

We further require the following: 

\begin{condition} \label{condition: real exact}
    $\Sigma$ is a {\em real exact} Lagrangian in $T^*C$, i.e., $\Sigma$ is exact with respect to $\lambda=\op{Re}\Lambda$, where $\Lambda=-p_z dz$ is the holomorphic Liouville form such that $d\Lambda=\Omega=dz\wedge dp_z$. 
\end{condition} 

See \cite[Section 4.5.1]{nho2023family} for a criterion/discussion of real exactness.

Let $HW(\sqcup_{i=1}^\kappa T^*_{q_i}C)$ be the wrapped HDHF algebra of the disjoint union $\sqcup_{i=1}^\kappa T^*_{q_i}C$ of $\kappa$ cotangent fibers over ${\bf q}=\{q_1,\dots,q_\kappa\}$.
Our main result is the following:  

\begin{theorem} \label{thm: map of algebras holomorphic version}
    There is a map of algebras
    $$\mathcal{F}_{\op{hol}}\colon HW(\sqcup_{i=1}^\kappa T^*_{q_i}C)\to \op{Mat}(N^{\kappa},\op{BSk}_\kappa(\Sigma)),$$
    which is given by a holomorphic curve count.
\end{theorem}

We also discuss enhancements of $\mathcal{F}_{\op{hol}}$ by introducing some homological variables in Section~\ref{section: homological variables}, bringing us closer to the works \cite{gaiotto2013} and \cite{neitzke2020q}. 

The isomorphism $HW(\sqcup_{i=1}^\kappa T^*_{q_i}C)\cong \op{BSk}_{\kappa}(C)$ from \cite{honda2024jems} then implies the following:

\begin{corollary}
    There is a map of algebras
    \begin{equation*}
        \mathcal{F}\colon \op{BSk}_{\kappa}(C) \to \op{Mat}(N^{\kappa},\op{BSk}_\kappa(\Sigma)).
    \end{equation*}
\end{corollary}

The main advantage of the map $\mathcal{F}$ is that the HOMFLY skein relations are naturally preserved under $\mathcal{F}$. On the other hand, the connection between the curve counting and the spectral network associated to the spectral curve $\Sigma_{\phi}$ is not evident from the definition.

In order to remedy this, we present a hybrid Floer-Morse version of the above map:

\begin{proposition}
    There is a map 
    $$\mathcal{F}_{\op{Mor}}\colon HW(\sqcup_{i=1}^\kappa T^*_{q_i}C)\to \op{Mat}(N^{\kappa},\op{BSk}_\kappa(\Sigma))$$
    which is given by counting holomorphic curves coupled with folded Morse flow trees.
\end{proposition}

\begin{conjecture}
\label{conj-morse}
    $\mathcal{F}_{\op{Mor}}=\mathcal{F}_{\op{hol}}$.
\end{conjecture}

In this paper we sketch a proof of Conjecture~\ref{conj-morse}, using a variant of the adiabatic limit theorems of  Fukaya-Oh \cite{fukaya1997zero} and Ekholm  \cite{ekholm2007morse}, which relate holomorphic curves and Morse flow trees. The details are left to future work, but {\em all the steps of the proof have similar/analogous results that have already appeared in the literature.}

The folded Morse flow trees considered in the definition of $\mathcal{F}_{\op{Mor}}$ are very similar to the M2-branes interpretation in the $q$-nonabelianization \cite{neitzke2020q} for $N=2$. The main difference is that we consider braids here, whereas Neitzke and Yan consider arcs in general. For $N=2$ our folded Morse tree count gives what \cite{neitzke2020q} calls {\em detour moves} and {\em exchange moves}.
The curve counts should be different for $N=3$ (in general we count fewer curves) due to the presence of the cylindrical direction. It would be interesting to explicitly compare this difference. 

\begin{remark}
    Although both the skein algebra and the braid skein algebra concern curves in a thickened surface, a direct connection between them is not evident to the authors. One tricky aspect concerns the framing. There is no canonical framing for elements in the braid skein algebra; see \cite[Section 1.3 and the proof of Theorem 4.1]{morton2021dahas} for a more detailed discussion. 
\end{remark}

The paper is organized as follows: 
Sections~\ref{section: skein algebras} and \ref{section: holomorphic symplectic to real symplectic} are preparatory.
In Section~\ref{section: the Floer-theoretic approach}, we define the Floer-theoretic map $\mathcal{F}_{\op{hol}}$ from HDHF of cotangent fibers of $C$ to the matrix-valued braid skein algebra and prove Theorem~\ref{thm: map of algebras holomorphic version}.  We discuss enhancements of $\mathcal{F}_{\op{hol}}$ by introducing some homological variables in Section~\ref{section: homological variables}.
In Section~\ref{section: hybrid Floer-Morse}, we present the hybrid Floer-Morse approach to $\mathcal{F}_{\op{hol}}$ and, in Section~\ref{section: equivalence of two approaches}, we sketch the proof that the two approaches coincide.

\vskip0.1in
\noindent \textit{Acknowledgements}. 
We thank Vivek Shende for suggesting this project, and Tobias Ekholm, Yoon Jae Nho, Francis Bonahon, Helen Wong, Vijay Higgins, and Octav Cornea for helpful conversations.  We also thank Erkao Bao for extensive conversations and help on gradient graph gluing.

\section{The braid skein algebra and the matrix-valued braid skein algebra} \label{section: skein algebras}

In this section we briefly review the braid skein algebra and introduce the matrix-valued braid skein algebra.

Let $\mathbf{q}=\{q_1,\dots,q_{\kappa}\}\subset C$ be a $\kappa$-tuple of basepoints. Let $\mathrm{Br}_{\kappa}(C,\mathbf{q})$ be the braid group on $C\times[0,1]$ consisting of $\kappa$ strands from ${\bf q}\times\{0\}$ to ${\bf q}\times\{1\}$ that are transverse to the slices $C\times\{t\}$, $t\in[0,1]$; the product is given by the concatenation of braids.

\begin{definition}[Morton-Samuelson]
    \label{def-skein}
    The braid skein algebra $\mathrm{BSk}_\kappa(C,\mathbf{q})$ is the quotient of the group algebra $\mathbb{Z}[\hbar][\mathrm{Br}_{\kappa}(C,\mathbf{q})]$ by the braid isotopy relation and the local HOMFLY skein relation:
    \begin{equation}
        \label{eq-skein'}
        \includegraphics[width=1cm,valign=c]{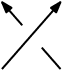}-\includegraphics[width=1cm,valign=c]{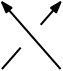}=\hbar \includegraphics[width=1cm,valign=c]{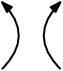}.
    \end{equation}
\end{definition}

See \cite{honda2024jems} for more details. 

Let $\pi: \Sigma \to C$ be an $N$-fold branched cover with simple ramification points, $Z\subset C$ be the set of branch points, $\widetilde Z\subset \Sigma$ be the set of ramification points sitting over $Z$, $P\subset C$ be the set of poles of $P_\phi$, and $\widetilde P$ be the set of punctures that complete $\Sigma$ to a closed Riemann surface (see Remark~\ref{remark: punctures}). Assume that ${\bf q}\cap (Z\cup P)=\varnothing$ and there is a single ramification point over each branch point.
Let $\pi^{-1}(q_i)=\{{p_{i1},\dots,p_{iN}}\}$ for $i=1,\dots,\kappa$. 
Consider an indexing set of maps
\begin{equation}\label{eqn: I}
    \mathfrak{I}=\{\alpha:\{1,\dots,\kappa\} \to \{1,\dots,N\}\}.
\end{equation}
We have $|\mathfrak{I}|=N^{\kappa}$. 
Given $\alpha \in \mathfrak{I}$, let $\mathbf{p_{\alpha}}=\{p_{1\alpha(1)},\dots,p_{\kappa\alpha(\kappa)}\}$ be a collection of $\kappa$ disjoint points on $\Sigma$ such that $\pi:\mathbf{p_{\alpha}}\to\mathbf{q}$ is a bijection. 

Define the braid skein algebra $\mathrm{BSk}_\kappa(\Sigma,\mathbf{p_{\alpha}},\mathbf{p_{\beta}})$ to be the quotient of the free $\mathbb{Z}[\hbar]$-module generated by braids in $\Sigma \times [0,1]$ from $\mathbf{p_{\alpha}}$ to $\mathbf{p_{\beta}}$, modulo the local HOMFLY skein relation (\ref{eq-skein'}) and an additional relation:
\begin{equation}
    \label{eq-skein-branch}
    \includegraphics[width=5cm,valign=c]{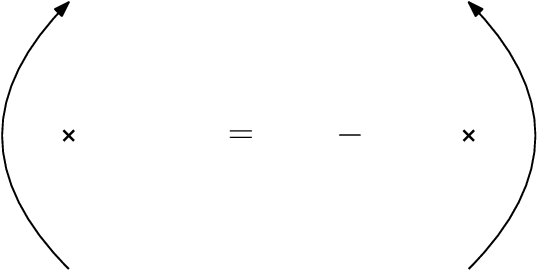}
\end{equation}
where the crossing denotes a simple ramification point on $\Sigma$. 

\begin{definition} \label{defn: matrix-valued braid skein algebra}
The {\em matrix-valued braid skein algebra on $\Sigma$ with respect to $\pi: \Sigma \to C$} is:
$$\op{Mat}(N^{\kappa},\op{BSk}_\kappa(\Sigma))=\{M=[M_{\alpha \beta}] ~|~ M_{\alpha \beta} \in \mathrm{BSk}_\kappa(\Sigma,\mathbf{p_{\alpha}},\mathbf{p_{\beta}}), \alpha,\beta \in \mathfrak{I}\},$$
where the product is given by $N^\kappa\times N^\kappa$-matrix multiplication and concatenation of braids.      
\end{definition}

\begin{remark} \label{remark: explanation of eq-skein-branch}
Relation \eqref{eq-skein-branch} arises as follows: Let $\mathfrak{s}$ be a spin structure on $C$ and let $\pi^*\mathfrak{s}$ be the pullback spin structure to $\Sigma\setminus \widetilde Z$. Note that $\pi^*\mathfrak{s}$ does not extend over the ramification points; this is evident from viewing a spin structure as a stable trivialization of $T\Sigma$ over the $1$-skeleton that extends over the $2$-skeleton. Suppose we orient the relevant moduli space of holomorphic curves $\mathcal{T}:={\mathcal{T}}^{\operatorname{ind}=1}(\mathbf{p}_\alpha,\mathbf{y}',\mathbf{y},\mathbf{p}_\beta)$ (see Section~\ref{subsection: proof of thm map of algebras holomorphic version}) using $\pi^*\mathfrak{s}$.  Consider a $1$-parameter family $u_\tau\in \mathcal{T}$, $\tau\in[-\epsilon,\epsilon]$, for $\epsilon>0$ small. There is an evaluation map $\mathcal{E}$ such that $\mathcal{E}(u_\tau)\in\op{Mat}(N^{\kappa},\op{BSk}_\kappa(\Sigma))$ (see Section~\ref{subsection: defn of F hol}).  If $\mathcal{E}(u_{-\epsilon})$ is locally given by the left-hand side of \eqref{eq-skein-branch}, $\mathcal{E}(u_{\epsilon})$ is locally given by the right-hand side, and the family $\mathcal{E}(u_\tau)$ crosses the ramification point transversely at $\tau=0$, then $u_{0}$ is not oriented and the orientation reverses as we cross $u_{0}$.  This gives \eqref{eq-skein-branch}.
\end{remark}

\begin{remark}
    One can also equally well choose a spin structure that is defined over all of $\Sigma$, in which case we switch $-$ to $+$ in \eqref{eq-skein-branch}.
\end{remark}

\section{From \texorpdfstring{$T^*_{\op{hol}}C$}{T*holC} to \texorpdfstring{$T^* C$}{T*C}} \label{section: holomorphic symplectic to real symplectic}

In this section we summarize how to pass from the holomorphic symplectic manifold $T^*_{\op{hol}}C$ to the real cotangent bundle $T^* C$.  The identification is locally given by $(z,p_z)\mapsto (x,y, p_x,-p_y)$.  
  
The canonical holomorphic symplectic form on $T^*_{\op{hol}}C$ is $\Omega=dz\wedge dp_z$ and the canonical symplectic form on $T^*C$ is 
\begin{equation*}
    \omega=\op{Re}\Omega=  \op{Re} dz\wedge dp_z=\op{Re} (dx+idy) (dp_x-i dp_y)= dx dp_x + dy dp_y.
\end{equation*}
We also have the holomorphic Liouville form $\Lambda=-p_z dz$ such that $d\Lambda=\Omega$ and the real Liouville form $\lambda =\op{Re}\Lambda=-(p_xdx+p_ydy)$.

Since $\Sigma:=\Sigma_\phi$ is a smooth algebraic curve in $T^*_{\op{hol}}C$, the holomorphic symplectic form $\Omega$ vanishes on $\Sigma$, that is, $\Sigma$ is a holomorphic Lagrangian submanifold of $T^*_{\op{hol}}C$.   Viewed as a surface in $T^*C$, the spectral curve $\Sigma$ becomes a real Lagrangian submanifold in $T^*C$ since $\omega=\op{Re}\Omega$ vanishes. [Verification: Letting $\{v, Jv\}$ be a basis for the tangent space $T\Sigma$ at a point, where $J$ is the complex structure for $T^*_{\op{hol}}C$, 
\begin{align*}
    dz dp_z(v,Jv)=dz(v) dp_z(Jv) - dz(Jv) dp_z(v)= i dz(v) dp_z(v) - i dz(v) dp_z(v)=0,
\end{align*}
since $dz$ and $dp_z$ are $(J,i)$-antilinear.]

Let $U$ be a contractible open subset of $C\setminus Z$, where $Z$ is the branch locus of $C$. Then $\pi^{-1}(U)$ can be written locally as the union of graphs $df_1,\dots, df_N$ of functions $U\to \R$. We explain the relationship between the locally defined functions $f_j$ and the defining polynomial $P_\phi(z,p_z)$ for $\Sigma$ from the introduction: On $\pi^{-1}(U)$, we can factor $P_\phi(z,p_z)= \prod_{j=1}^N(p_z-\psi_j(z))$.  Then $\Sigma$ is locally given by $p_z=\psi_j(z)=a_j(z) +i b_j(z)$. Writing $z=x+iy$ and $p_z=p_x-ip_y$, in real terms $\Sigma$ is the graph of $\alpha_j:=a_j dx-b_jdy=\op{Re}\psi_j(z)dz$.  Then, by the Cauchy-Riemann equations, $d\alpha_j=0$ and we can write $\alpha_j= df_j$. The foliation given by $\op{Re} \psi_j(z)$ is tangent to the level sets of $f_j$ and the foliation given by $\op{Im} \psi_j(z)dz$ is tangent to $\nabla f_j$.  The same holds for differences $\psi_i-\psi_j$ and $f_i-f_j$.

\section{The Floer-theoretic approach} \label{section: the Floer-theoretic approach}

Let $C$ be a closed surface. Let $CW(\sqcup_{i=1}^\kappa T^*_{q_i}C)$ be the wrapped HDHF $A_\infty$-algebra of $\kappa$ disjoint cotangent fibers of $T^*C$ and let $HW(\sqcup_{i=1}^\kappa T^*_{q_i}C)$ be its (co-)homology.   We refer to \cite{honda2024jems} and \cite{colin2020applications} for the definition of $CW(\sqcup_{i=1}^\kappa T^*_{q_i}C)$, which in our case is an ordinary algebra supported in degree $0$.

The goal of this section is to prove Theorem~\ref{thm: map of algebras holomorphic version}, i.e., define a Floer-theoretic evaluation map $\mathcal{F}_{\op{hol}}$ from $CW(\sqcup_{i=1}^\kappa T^*_{q_i}C)$ to the matrix-valued braid skein algebra $\op{Mat}(N^{\kappa},\op{BSk}_\kappa(\Sigma))$ on $\Sigma$, where $\pi:\Sigma\to C$ is an $N$-fold branched cover. Recalling the isomorphism 
\begin{equation} \label{eqn: isom with bsk}
    HW(\sqcup_{i=1}^\kappa T^*_{q_i}C)\cong \op{BSk}_{\kappa}(C,{\bf q})
\end{equation} 
from \cite{honda2024jems}, we then obtain a map $\mathcal{F}: \op{BSk}_{\kappa}(C,{\bf q})\to \op{Mat}(N^{\kappa},\op{BSk}_\kappa(\Sigma))$.  We also observe that the proof of Theorem~\ref{thm: map of algebras holomorphic version} follows the same outline as that of \eqref{eqn: isom with bsk}.

\subsection{The base direction} \label{subsection: base direction}

The key idea to defining HDHF is to count holomorphic maps of arbitrary genus to a target symplectic manifold that also branch cover over a base (i.e., the domain of our map is viewed as an element of a Hurwitz space). 

Let $D$ be the unit disk in $\mathbb{C}$ and $D_m=D\setminus\{p_0,\dots,p_m\}$, where $p_i\in\partial D$ are boundary marked points arranged in counterclockwise order. Let $\partial_i D_m$ be the boundary arc from $p_i$ to $p_{i+1}$ for $i\in\mathbb{Z}/(m+1)\mathbb{Z}$. 
Let $\mathcal{A}_m$ be the moduli space of $D_m$ modulo automorphisms; we choose representatives $D_m$ of equivalence classes of $\mathcal{A}_m$ in a smooth manner and abuse notation by writing $D_m\in \mathcal{A}_m$.

As usual, there is a smooth choice of strip-like ends $e_i$, $i=0,\dots, m$, for each $D_m$.
We view the strip-like ends $e_i$, $i=1,\dots, m$, of $D_m$ corresponding to $p_i$ as positive ends $[0,\infty)_{r_i}\times[0,1]_{t_i}$ and view the strip-like end $e_0$ corresponding to $p_0$ as the negative end $(-\infty,0]_{r_0}\times[0,1]_{t_0}$. 

Next, let $D_{m,l}$ be $D_m$ together with a finite (unordered) set $\{h_1,\dots,h_l\}$ of interior marked points. We denote the set of $D_{m,l}$ modulo automorphisms by $\mathcal{A}_{m,l}$.

\begin{convention}\label{conv: omitting subscripts}
    Sometimes we omit the subscript $l$ of $D_{m,l}$ and write $D_m$ to indicate the result of applying a forgetful functor.
\end{convention}

\begin{lemma}[Smooth choice of strip-like ends and necks]\label{lemma: strip-like ends and necks} $\mbox{}$ 
    \be
    \item There is a smooth choice of strip-like ends $e_i$, $i=0,\dots, m$, for each $D_{m,l}$, on which there are no marked points.
    \item For $D_{m,l}\in \mathcal{D}_{m,l}$ such that:
    \be
    \item $D_m$ is sufficiently close to the boundary $\mathcal{D}_{m}$ or 
    \item $D_{m,l}$ has marked points sufficiently far into some $e_i$,
    \ee
    there is a collection of smoothly chosen strip-like necks $n_i$ of type $[-L,L]_r\times[0,1]_t\subset D_{m,l}$ (here $L$ depends on $D_{m,l}$ and $n_i$ and $([-L,L]\times[0,1])\cap \bdry D_{m,l}=[-L,L]\times\{0,1\}$), on which there are no marked points.
    \item Each component of $D_{m,l}\setminus ((\cup_i n_i)\cup (\cup_j e_j))$ agrees with some $D_{m',l'}\setminus (\cup_j e_j)$, where $m+l> m'+l'$. 
    \ee
\end{lemma}   

For convenience, in this paper we also write $Y_{m-1}:=D_m$, where $\partial_i Y_{m-1} = \partial_i D_m$ for $i=0,\dots, m$. 
Let $\mathcal{Y}_{m-1}$ be the moduli space of $Y_{m-1}$ modulo automorphisms; again we may choose representatives $Y_{m-1}$ of equivalence classes of $\mathcal{Y}_{m-1}$ in a smooth manner. We similarly define $Y_{m-1,l}$ and $\mathcal{Y}_{m-1,l}$. In this paper, we will mainly use $\mathcal{Y}_1$ and $\mathcal{Y}_2$.

\begin{figure}[ht]
    \centering
    \includegraphics[width=14cm]{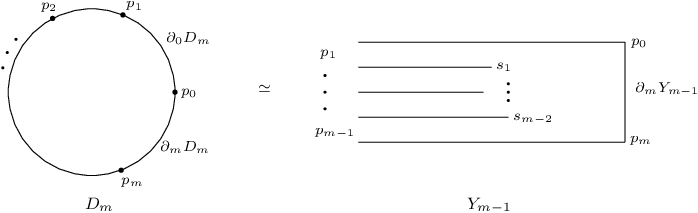}
    \caption{}   
    \label{fig: base}
\end{figure}

Consider the following representative of $Y_{m-1}$ for $m\geq 2$: 
\begin{equation} \label{eqn: model for Ym-1}
    (-\infty,0]_s\times[0,1]_\tau\subset\mathbb{C}
\end{equation}
with $m-2$ slits along $\tau=(m-2)/(m-1),\dots,\tau=1/(m-1)$; here $s=\op{Re}z$ and $\tau=\op{Im}z$.  See the right-hand side of Figure \ref{fig: base}. Since $Y_{m-1}$ admits no nontrivial automorphism group, we call this presentation {\em the standard coordinate} of $Y_{m-1}$. The coordinates of $\mathcal{Y}_{m-1}$ are given by the length of the slits.

\subsection{Consistent collection of perturbation data} \label{subsection: consistent collection 1}

In order to achieve the transversality of the moduli spaces of holomorphic maps used in the definition $CW(\sqcup_{i=1}^\kappa T^*_{q_i}C)$, it suffices to consider domain-dependent almost complex structures that are parametrized by $Y_1$ and $Y_2$.  However, in order to compare with the hybrid Floer-Morse approach in Section~\ref{section: equivalence of two approaches}, we will require a finer perturbation of almost complex structures which depends on $Y_{1,l}$ and $Y_{2,l}$, $l\geq0$, which we describe in this subsection.

Fix a generic Riemannian metric $g_0$ on the closed surface $C$ and let $|\cdot|$ be the induced norm on $T^*C$. 
Let $\mathcal{H}(T^*C)\subset C^\infty(T^*C,\mathbb{R})$ be the set of smooth functions $H$ such that 
\begin{equation*}
    H(q,p)=\frac{1}{2}|p|^2
\end{equation*}
outside some compact subset of $T^*C$. Fix $\epsilon>0$ and let $\mathcal{H}^\epsilon(T^*C)\subset\mathcal{H}(T^*C)$ be the subset such that $H(q,p)=0$ for $|p|<\epsilon$.
Also fix $H_0\in\mathcal{H}(T^*C)$ such that all time-$1$ Hamiltonian orbits between the cotangent fibers considered in this paper are nondegenerate.

Consider a smooth family of Hamiltonian functions parametrized by $Y_1$:
\begin{equation}
    H_{Y_1}\colon Y_1\to \mathcal{H}(T^*C),
\end{equation}
which takes values in $\mathcal{H}^\epsilon(T^*C)$ near $\partial_2 Y_1$. We require that $H_{Y_1}=H_0$ on the strip-like end $e_1$.
Write $X_{H_{Y_1}}$ for the Hamiltonian vector field of $H_{Y_1}$, i.e., $\omega(X_{H_{Y_1}},\cdot)=dH_{Y_1}$.

Denote by $\mathcal{J}(T^*C,\omega)$ the set of almost complex structures on $T^*C$ which are compatible with the canonical symplectic form $\omega$. 
Fix a smooth family $J_t\in\mathcal{J}(T^*C,\omega)$ for $t\in[0,1]$.
Let
\begin{equation}
   \mathcal{J}_{1,l}:=\{ J_{Y_{1,l}}\colon Y_{1,l}\to\mathcal{J}(T^*C,\omega)~|~ Y_{1,l}\in \mathcal{Y}_{1,l}\}
\end{equation}
be a family of almost complex structures that depends smoothly on $Y_{1,l}$ and such that $J_{Y_{1,l}}(r_1,t_1)=J_{t_1}$ on the strip-like end $e_1$. 

\begin{definition}\label{defn: perturbation data for Y1l}
    A pair $(H_{Y_1},\{J_{Y_{1,l}}\})$ satisfying the above is called a {\em perturbation data for $\mathcal{Y}_{1,l}$.}
\end{definition}

If $\mathcal{J}_{1,l}$ is consistent with the boundary strata of $\mathcal{Y}_{1,l}$ (by this we mean each $J_{Y_{1,l}}$ is $r$-invariant on the strip-like ends and necks given by Lemma~\ref{lemma: strip-like ends and necks} and respects the identifications given by Lemma~\ref{lemma: strip-like ends and necks}(3)), then $\{\mathcal{J}_{1,l}~|~l\geq0\}$ is called a {\em consistent collection of almost complex structures}.

\begin{lemma} \label{lemma: existence of consistent collection for J}
There exists a consistent collection $\{\mathcal{J}_{1,l}~|~l\geq0\}$ of almost complex structures. 
\end{lemma}

\begin{proof}
    This is an easy inductive construction, which uses the fact that $\mathcal{J}(T^*C,\omega)$ is contractible.
\end{proof}

\begin{definition}
    A {\em consistent collection of perturbation data for $\{\mathcal{Y}_{1,l}\}$} consists of perturbation data for each $\mathcal{Y}_{1,l}$ such that $\{\mathcal{J}_{1,l}\}$ is a consistent collection of almost complex structures.
\end{definition}

\subsection{Moduli spaces and transversality}

Given $\mathbf{y}\in CW(\sqcup_{i=1}^\kappa T^*_{q_i}C)$ and $\alpha,\beta\in \mathfrak{I}$, where $\mathfrak{I}$ is given by \eqref{eqn: I}, let $\mathcal{T}^l(\mathbf{p}_\alpha,\mathbf{y},\mathbf{p}_\beta)$ be the set of $((F,j),u)$, where $(F,j)$ is a compact Riemann surface with boundary, $\mathbf{w}_0,\mathbf{w}_1,\mathbf{w}_2$ are disjoint $\kappa$-tuples of boundary marked points of $F$, $\dot F=F\setminus\cup_i \mathbf{w}_i$, and
\begin{equation*}
    u\colon \dot{F}\to Y_1\times T^*C
\end{equation*}
is a continuous map such that
\begin{equation}
    (du-X_{H_{Y_1}}\otimes d\tau)^{0,1}_{j,J_{Y_{1,l}}}=0,
\end{equation}
and the following hold:
\begin{enumerate}
    \item For $m=0,1$, $\pi_{T^*C}\circ u(z)\in\sqcup_{i}T_{q_i}^*C$ if $\pi_{Y_1}\circ u(z)\subset\partial_m Y_1$. Each component of $\partial\dot F$ that projects to $\partial_m Y_1$ maps to a distinct $T^*_{q_i}C$.
    \item $\pi_{T^*C}\circ u(z)\in\Sigma$ if $\pi_{Y_1}\circ u(z)\subset\partial_2 Y_1$. \label{sigma}
    \item $\pi_{T^*C}\circ u$ tends to $\mathbf{y}$, $\mathbf{p}_\beta$ as $s_1,s_2\to +\infty$ and tends to $\mathbf{p}_\alpha$ as $s_0\to-\infty$.
    \item $\pi_{Y_1}\circ u$ is a $\kappa$-fold branched cover of $Y_1$, where we assume the branch points are simple and correspond to the marked points on $Y_{1,l}$.
\end{enumerate}
Here $\pi_{Y_l}$ and $\pi_{T^*C}$ are projections of $Y_l\times T^*C$ to the first and second factors.

We then set
$$\mathcal{T}(\mathbf{p}_\alpha,\mathbf{y},\mathbf{p}_\beta)=\sqcup_{l\geq0}\mathcal{T}^l(\mathbf{p}_\alpha,\mathbf{y},\mathbf{p}_\beta).$$
Note that if $((F,j),u)\in\mathcal{T}^l(\mathbf{p}_\alpha,\mathbf{y},\mathbf{p}_\beta)$, then $\chi(F)=\kappa-l$ by the Riemann-Hurwitz formula.

\begin{figure}[ht]
    \centering
    \includegraphics[width=10cm]{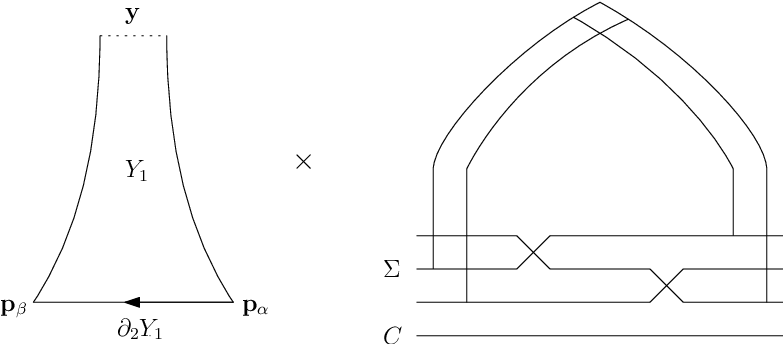}
    \caption{}
    \label{fig: F-hol}
\end{figure}

\begin{lemma} \label{lemma-dim}
    For a generic consistent collection of perturbation data for $\{\mathcal{Y}_{1,l}\}$, $\mathcal{T}(\mathbf{p}_\alpha,\mathbf{y},\mathbf{p}_\beta)$ is of dimension $0$ and consists of discrete regular curves for all $\mathbf{p}_\alpha$, $\mathbf{y}$ and $\mathbf{p}_\beta$.
\end{lemma}

\begin{proof}
    The virtual dimension of $\mathcal{T}(\mathbf{p}_\alpha,\mathbf{y},\mathbf{p}_\beta)$ is $0$ by an index computation similar to that of \cite[Theorem 2.3]{honda2024jems}.
    The transversality of $\mathcal{T}(\mathbf{p}_\alpha,\mathbf{y},\mathbf{p}_\beta)$ is standard due to the flexible choice of almost complex structures.
\end{proof}

\begin{lemma}
    \label{lemma-cpt}
    For a generic consistent collection of perturbation data for $\{\mathcal{Y}_{1,l}\}$ and fixed $\mathbf{p}_\alpha$, $\mathbf{y}$, $\mathbf{p}_\beta$, and $l\geq0$, the moduli space $\mathcal{T}^l(\mathbf{p}_\alpha,\mathbf{y},\mathbf{p}_\beta)$ consists of finitely many curves.
\end{lemma}

\begin{proof} 
    This is standard Gromov compactness since there is an energy bound for curves in $\mathcal{T}^l(\mathbf{p}_\alpha,\mathbf{y},\mathbf{p}_\beta)$. 
\end{proof}

\subsection{Definition of the map \texorpdfstring{$\mathcal{F}_{\op{hol}}$}{Fhol}} \label{subsection: defn of F hol}

Given $u\in\mathcal{T}(\mathbf{p}_\alpha,\mathbf{y},\mathbf{p}_\beta)$, the restriction of $\pi_{T^*C} \circ u$ to $(\pi_{Y_1}\circ u)^{-1}(\partial_2 Y_1)$ gives an element $\mathcal{E}(u)\in\op{Mat}(N^{\kappa},\op{BSk}_\kappa(\Sigma))$ --- called the {\em evaluation map of $u$} --- which is supported in the $\alpha\beta$-entry of the $N^\kappa\times N^\kappa$-matrix.  Here we are fixing a parametrization $\phi_{Y_1}:[0,1]\to \bdry_2 Y_1$.

We then define
\begin{gather}\label{eqn: Fhol}
    \mathcal{F}_{\op{hol}}\colon CW(\sqcup_{i=1}^\kappa T^*_{q_i}C)\to \op{Mat}(N^{\kappa},\op{BSk}_\kappa(\Sigma)),\\
    \nonumber\mathbf{y}\mapsto\sum_{u\in\mathcal{T}(\mathbf{p}_\alpha,\mathbf{y},\mathbf{p}_\beta)} (-1)^{\sharp(u)}\cdot\hbar^{\kappa-\chi(F)}\cdot\mathcal{E}(u).
\end{gather}
Here $\sharp(u)\in\mathbb{Z}/2\mathbb{Z}$ denotes the orientation of $u$, which depends on the spin structure on $\Sigma$. 
A canonical choice of spin structure on $\Sigma$ is induced from the spin structure on $C$, which necessarily requires the additional relation (\ref{eq-skein-branch}). See \cite[Section 6.2.2]{nho2023family} for details.
Since $CW(\sqcup_{i=1}^\kappa T^*_{q_i}C)$ is supported in degree $0$, $\mathcal{F}_{\op{hol}}$ is defined on the cohomology level.

\subsection{Proof of Theorem~\ref{thm: map of algebras holomorphic version}} \label{subsection: proof of thm map of algebras holomorphic version}

In this subsection we prove Theorem~\ref{thm: map of algebras holomorphic version}. The proof is similar to that of \cite[Proposition 6.5]{honda2024jems} and we show that
\begin{equation}\label{eqn: algebra map}
    \mathcal{F}_{\op{hol}}(\mu^2(\mathbf{y},\mathbf{y}'))=\mathcal{F}_{\op{hol}}(\mathbf{y})\mathcal{F}_{\op{hol}}(\mathbf{y}'),
\end{equation}
for any $\mathbf{y},\mathbf{y}'\in CW(\sqcup_{i}T_{q_i}^*C)$, where $\mu^2$ is the product map of HDHF and the right-hand side of Equation~\eqref{eqn: algebra map} is a concatenation written in composition notation.

First observe the following: let $\psi^\rho$ be the time-$\op{log}\rho$ flow of the Liouville vector field of $pdq$. Since $(\psi^\rho)^*\omega=\rho\omega$, following \cite[Lemma 3.5]{abouzaid2012wrapped} we have:

\begin{lemma}
    There is an isomorphism
    \begin{equation*}
        CW^*(\boldsymbol{L_0},\boldsymbol{L_1};H, J_t)\simeq CW^*(\psi^\rho(\boldsymbol{L_0}),\psi^\rho(\boldsymbol{L_1});\tfrac{H}{\rho}\circ\psi^\rho, \psi^\rho_* J_t).
    \end{equation*}
\end{lemma}

In order to define the relevant moduli spaces $\mathcal{T}^l(\mathbf{p}_\alpha,\mathbf{y}',\mathbf{y},\mathbf{p}_\beta)$, we first construct a {\em consistent collection $(\{H_{Y_m}\},\{J_{Y_{m,l}}\}, \{\alpha_{Y_m}\})$ of perturbation data for $\{\mathcal{Y}_{m,l}\}_{m\leq 2}$}, following the procedure of \cite[Section 4]{abouzaid2012wrapped}.  We have already treated the $m=1$ case, as $\alpha_{Y_1}$ is not necessary.

Consider a family of Hamiltonian functions parametrized by $Y_2$:
\begin{equation}
    H_{Y_2}\colon Y_2\to \mathcal{H}(T^*C),
\end{equation}
which takes values in $\mathcal{H}^\epsilon(T^*C)$ near $\partial_3 Y_2$. We further require the following:
\begin{enumerate}
    \item $H_{Y_2}$ coincides with $H_{Y_1}$ on each of the two pieces of $Y_1$ in the right-hand degeneration of Figure \ref{fig: F-algebra};
    \item referring to the left-hand degeneration of Figure \ref{fig: F-algebra}, where $Y_2$ degenerates into A and B,  $H_{Y_2}=H_0$ on the strip-like ends $e_1,e_2$ of A,
    \begin{equation*}
        H_{Y_2}=({H_0}\circ\psi^2)/4
    \end{equation*}
    on the strip-like end $e_0$ of A, and $H_{Y_2}=(H_{Y_1}\circ\psi^2)/4$ on B.
\end{enumerate}

The choice of $J_{Y_{2,l}}$ is similar and coincides with $J_{t_1}$ and $J_{t_2}$ on the strip-like ends $e_1,e_2$. We further require that 
\begin{enumerate}
    \item $J_{Y_{2,l}}=J_{Y_1}$ on the two copies of $Y_1$ of the right-hand side degeneration of Figure \ref{fig: F-algebra};
    \item $J_{Y_{2,l}}=\psi^2_*J_{t_0}$ on the strip-like end $e_0$ of A and $J_{Y_{2,l}}=\psi^2_*J_{Y_{1,l'}}$ on $\mathrm{B}=Y_{1,l'}$.
\end{enumerate}

The last piece of perturbation data is a closed $1$-form $\alpha_{Y_2}$ on $Y_2$ such that
\begin{enumerate}
    \item $\alpha_{Y_2}=d\tau$ on the two copies of $Y_1$ of the right-hand side degeneration of Figure \ref{fig: F-algebra};
    \item $\alpha_{Y_2}=2d\tau$ near $\partial_3 Y_2$;
    \item $\alpha_{Y_2}=dt_i$ on $e_i$, $i=1,2$, and $\alpha_{Y_2}=2dt_0$ on the end $e_0$ of A and $\alpha_{Y_2}=2d\tau$ on B.
\end{enumerate}

The existence of $(H_{Y_2},\{J_{Y_{2,l}}\},\alpha_{Y_2}\})$ satisfying the above is straightforward. 

\s
Following the terminology of \cite[Section 3.5]{abouzaid2012wrapped}, we say the ends $e_1,e_2$ are of {\em weight} $1$ and $e_0$ is of {\em weight} $2$.

Now we define the moduli space $\mathcal{T}^l(\mathbf{p}_\alpha,\mathbf{y}',\mathbf{y},\mathbf{p}_\beta)$ as the set of $(F,u)$, where $\mathbf{y},\mathbf{y}'\in CW(\sqcup_{i=1}^\kappa T^*_{q_i}C)$, $(F,j)$ is a compact Riemann surface with boundary, $\mathbf{w}_0,\dots,\mathbf{w}_3$ are disjoint $\kappa$-tuples of boundary marked points of $F$, $\dot F=F\setminus\cup_i \mathbf{w}_i$, and
\begin{equation*}
    u\colon \dot{F}\to Y_2\times T^*C
\end{equation*}
is a continuous map such that
\begin{equation}
    (du-X_{H_{Y_2}}\otimes \alpha_{Y_2})^{0,1}_{j,J_{Y_{2,l}}}=0,
\end{equation}
and the following hold:
\begin{enumerate}
    \item For $m=0,1,2$, $\pi_{T^*C}\circ u(z)\in\sqcup_{i}T_{q_i}^*\Sigma$ if $\pi_{Y_2}\circ u(z)\subset\partial_m Y_2$. Each component of $\partial\dot F$ that projects to $\partial_m Y_2$ maps to a distinct $T^*_{q_i}C$.
    \item $\pi_{T^*C}\circ u(z)\in\Sigma$ if $\pi_{Y_2}\circ u(z)\subset\partial_3 Y_2$.
    \item $\pi_{T^*C}\circ u$ tends to $\mathbf{y}'$, $\mathbf{y}$, $\mathbf{p}_\beta$ as $s_1,s_2,s_3\to +\infty$ and tends to $\mathbf{p}_\alpha$ as $s_0\to-\infty$.
    \item $\pi_{Y_2}\circ u$ is a $\kappa$-fold branched cover of $Y_2$, where the simple branched values correspond to the marked points on $Y_{2,l}$.
\end{enumerate}
We define $\mathcal{T}(\mathbf{p}_\alpha,\mathbf{y}',\mathbf{y},\mathbf{p}_\beta)=\sqcup_{l\geq0}\mathcal{T}^l(\mathbf{p}_\alpha,\mathbf{y}',\mathbf{y},\mathbf{p}_\beta)$.

\begin{lemma} \label{lemma-dim-2}
    For a generic consistent collection of perturbation data for $\{\mathcal{Y}_{m,l}\}_{m\leq 2}$, $\mathcal{T}(\mathbf{p}_\alpha,\mathbf{y}',\mathbf{y},\mathbf{p}_\beta)$ is a smooth manifold of dimension $1$ and admits a compactification.
\end{lemma}

\begin{proof}
    The virtual dimension of $\mathcal{T}(\mathbf{p}_\alpha,\mathbf{y}',\mathbf{y},\mathbf{p}_\beta)$ is $1$ by an index formula similar to that of \cite[Theorem 2.3]{honda2024jems}.  The transversality and compactness of $\mathcal{T}(\mathbf{p}_\alpha,\mathbf{y}',\mathbf{y},\mathbf{p}_\beta)$ are standard. Note that for the purposes of this lemma we may use an arbitrary spin structure that is defined on all of $\Sigma$. 
\end{proof}

Next, using the pullback spin structure $\pi^*\mathfrak{s}$ on $\Sigma\setminus \widetilde Z$, we consider all possible codimen\-sion-$1$ boundary degenerations of $\mathcal{T}(\mathbf{p}_\alpha,\mathbf{y}',\mathbf{y},\mathbf{p}_\beta)$. 
There are four types of boundary degenerations:
    \begin{enumerate}
        \item[(T1)] $\coprod_{\mathbf{y}'',\chi'+\chi''-\kappa=\chi} \mathcal{M}^{\operatorname{ind}=0,\chi'}(\mathbf{y}',\mathbf{y},\mathbf{y}'') \times\mathcal{T}^{\operatorname{ind}=0,\chi''}(\mathbf{p}_\alpha,\mathbf{y}'',\mathbf{p}_\beta)$;
        \item[(T2)] $\coprod_{\mathbf{p}_\gamma,\chi'+\chi''-\kappa=\chi}\mathcal{T}^{\operatorname{ind}=0,\chi'}(\mathbf{p}_\gamma,\mathbf{y},\mathbf{p}_\beta)\times\mathcal{T}^{\operatorname{ind}=0,\chi''}(\mathbf{p}_\alpha,\mathbf{y}',\mathbf{p}_\gamma)$;
        \item[(T3)] the set $\partial_n \overline{\mathcal{T}}^{\operatorname{ind}=1, \chi}(\mathbf{p}_\alpha,\mathbf{y}',\mathbf{y},\mathbf{p}_\beta)$ with a nodal degeneration along $\Sigma$;
        \item[(T4)] at some point the $\alpha\beta$-entry of $\mathcal{E}(u)$ (recall $\mathcal{E}(u)$ is an $N^\kappa\times N^\kappa$ matrix and is supported only in the $\alpha\beta$-entry) intersects a point in $\widetilde{Z}$.
    \end{enumerate}
    (T1) is given on the left-hand side of Figure \ref{fig: F-algebra} 
    and contributes to $\mathcal{F}_{\op{hol}}(\mu^2(\mathbf{y},\mathbf{y}'))$, where $\mathcal{M}^{\operatorname{ind}=0,\chi'}(\mathbf{y}',\mathbf{y},\mathbf{y}'')$ is the moduli space computing $\mu^2(\mathbf{y},\mathbf{y}')$ in HDHF. 
    (T2) is given on the right-hand side of Figure \ref{fig: F-algebra} and contributes to $\mathcal{F}_{\op{hol}}(\mathbf{y})\mathcal{F}_{\op{hol}}(\mathbf{y}')$.
    A standard gluing argument shows that all contributions to $\mathcal{F}_{\op{hol}}(\mu^2(\mathbf{y},\mathbf{y}'))$ and $\mathcal{F}_{\op{hol}}(\mathbf{y})\mathcal{F}_{\op{hol}}(\mathbf{y}')$ come from such broken degenerations. 
    (T3) is given by Figure \ref{fig: F-skein} and corresponds to the HOMFLY skein relation (\ref{eq-skein'}) which is described in \cite{ekholm2021skeins} and \cite{honda2024jems}.
    In Case (T4), the orientations of the curves before and after touching the ramification point differ by $-1$ as explained in Remark~\ref{remark: explanation of eq-skein-branch}; this corresponds to Relation \eqref{eq-skein-branch}.

\begin{figure}[ht]
    \centering
    \includegraphics[width=13cm]{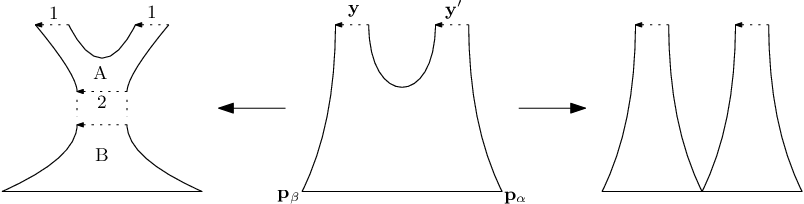}
    \caption{Two possible degenerations of ${\mathcal{T}}^{\operatorname{ind}=1, \chi}(\mathbf{p}_\alpha,\mathbf{y}',\mathbf{y},\mathbf{p}_\beta)$. The weights of ends of A are labeled.}
    \label{fig: F-algebra}
\end{figure}

\begin{figure}[ht]
    \centering
    \includegraphics[width=15cm]{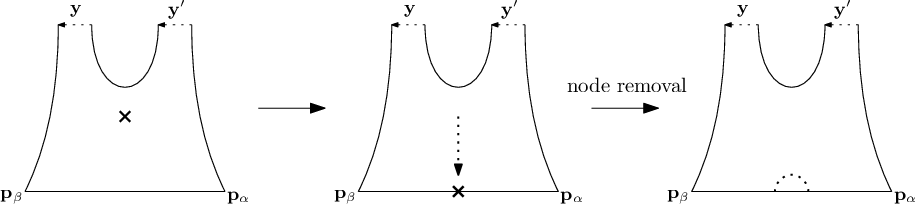}
    \caption{The nodal degeneration of ${\mathcal{T}}^{\operatorname{ind}=1, \chi}(\mathbf{p}_\alpha,\mathbf{y}',\mathbf{y},\mathbf{p}_\beta)$.}
    \label{fig: F-skein}
\end{figure}

\section{Modifications of \texorpdfstring{$\Sigma$}{Sigma}} \label{section: perturbing Sigma}

In order to pass from the Floer-theoretic approach to the hybrid Floer-Morse approach, we define two modifications of $\Sigma$.  

(1) Apply a $C^\infty$-small Hamiltonian perturbation near the ramification points on $\Sigma$ to convert branch points to cusp folds and swallowtails, which are more generic singularities; this is done in order to use Ekholm's adiabatic limit theorem \cite{ekholm2007morse} with minimal change. In the front projection of the Legendrian lift of $\Sigma$, a ramification point is modified to three cusp folds which intersect at three swallowtail singularities; see Figure \ref{fig: cusp}. 
\begin{figure}[ht]
    \centering
    \includegraphics[width=10cm]{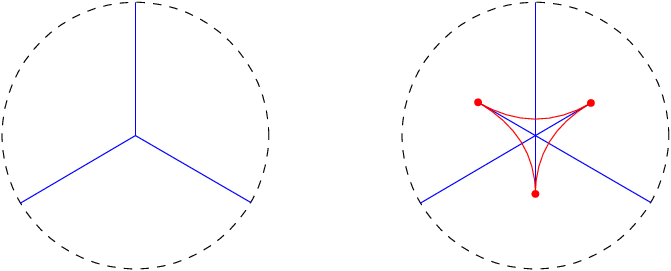}
    \caption{Left: the projection to $\mathbb{R}^2_{x,y}$ of the front of the Legendrian lift of $\Sigma$ near the ramification point. The blue rays denote the transversal intersection of two sheets.  Right: we apply a small Hamiltonian perturbation to $\Sigma$. The red edges are cusps and the red dots are swallowtail singularities.}
    \label{fig: cusp}
\end{figure}
We omit the details of this standard construction and instead refer the reader to \cite{casals22,casals2025spectralnetworksbettilagrangians}. Choose a sequence $\{\zeta_\nu\}_{\nu\in \Z_{>0}}$ of Hamiltonian perturbations of $\Sigma$ modeled on Figure~\ref{fig: cusp} such that the size of the ``cusp region'' (i.e., the region enclosed by the dashed circle on the right-hand side of Figure~\ref{fig: cusp}) tends to zero and denote $\Sigma_\nu\coloneqq\zeta_\nu(\Sigma)$.  More precisely, choose a decreasing sequence $\{R_\nu>0\}$ such that $R_1$ is small and $R_\nu\to0$ as $\nu\to\infty$. Fix a metric $g_0$ on $C$ and let $U_{R_\nu}$ be the union of all open $R_\nu$-disks in $C$ centered at points of $Z$.  We require the projection of the support of $\zeta_\nu$ and the folded region of $\Sigma_\nu$ to be inside $U_{R_\nu}$.

(2) Given a Lagrangian $\Sigma$ which is locally given by the graph $p_z=df(z)$ and $\epsilon>0$ small, we define $\Sigma^\epsilon$ as the Lagrangian locally given by $p_z=\epsilon df(z)$.  Also let $\pi_\epsilon:\Sigma^\epsilon\to \Sigma$ be the natural diffeomorphism, given by projecting along the fibers of $T^*C$.

\s
We define $\op{Mat}(N^{\kappa},\op{BSk}_\kappa(\Sigma_\nu^\epsilon))$ as in Definition~\ref{defn: matrix-valued braid skein algebra} by choosing the basepoints $\mathbf{q}$ away from $U_{R_1}$. Then there is a sequence of isomorphisms
\begin{equation*}
    (\zeta_\nu^{-1}\circ \pi_{\nu,\epsilon})_*:\op{Mat}(N^{\kappa},\op{BSk}_\kappa(\Sigma^\epsilon_\nu))\to \op{Mat}(N^{\kappa},\op{BSk}_\kappa(\Sigma)),
\end{equation*}
where $\pi_{\nu,\epsilon}$ is the natural diffeomorphism $\Sigma^\epsilon_\nu\to \Sigma_\nu$.

The spin structure on $\Sigma^\epsilon_\nu$ is induced from the one on $\Sigma$ via pushforward.

Let $\nu\in \Z_{>0}$ and $\epsilon>0$ small. We define $\mathcal{T}^l_{\nu,\epsilon}(\mathbf{p}_\alpha,\mathbf{y},\mathbf{p}_\beta)$ and $\mathcal{T}_{\nu,\epsilon}(\mathbf{p}_\alpha,\mathbf{y},\mathbf{p}_\beta)$ in the same way as $\mathcal{T}^l(\mathbf{p}_\alpha,\mathbf{y},\mathbf{p}_\beta)$ and $\mathcal{T}(\mathbf{p}_\alpha,\mathbf{y},\mathbf{p}_\beta)$, via the natural identifications of $\mathbf{p}_\alpha$, $\mathbf{p}_\beta$ on $\Sigma$ and $\mathbf{p}_\alpha$, $\mathbf{p}_\beta$ on $\Sigma^\epsilon_\nu$ and some abuse of notation. It is straightforward to verify that the analogs of Lemmas \ref{lemma-dim} and \ref{lemma-cpt} still hold for $\mathcal{T}^l_{\nu,\epsilon}(\mathbf{p}_\alpha,\mathbf{y},\mathbf{p}_\beta)$ and $\mathcal{T}_{\nu,\epsilon}(\mathbf{p}_\alpha,\mathbf{y},\mathbf{p}_\beta)$.

Then we define
\begin{gather}
    \mathcal{F}^{\nu,\epsilon}_{\op{hol}}\colon CW(\sqcup_{i=1}^\kappa T^*_{q_i}C)\to \op{Mat}(N^{\kappa},\op{BSk}_\kappa(\Sigma^\epsilon_\nu))
\end{gather}
in the same way as $\mathcal{F}_{\mathrm{hol}}$, by counting elements in $\mathcal{T}_{\nu,\epsilon}(\mathbf{p}_\alpha,\mathbf{y},\mathbf{p}_\beta)$.
All the lemmas from Section~\ref{section: the Floer-theoretic approach} also hold for $\mathcal{F}^{\nu,\epsilon}_{\op{hol}}$.

The proof of the following is similar to that of Lemma \ref{lemma-dim-2} and is omitted:

\begin{lemma}
    \label{lemma-limit-F}
    $\mathcal{F}_{\op{hol}}= (\zeta_\nu^{-1}\circ \pi_{\nu,\epsilon})_*\circ\mathcal{F}_{\op{hol}}^{\nu,\epsilon}$ for all $\nu\in \Z_{>0}$ and $\epsilon>0$.
\end{lemma}

\section{Homological variables}
\label{section: homological variables}

In this section we introduce the homological variable $c$ which extends the coefficient ring.

{\em Assume now that $C$ is a closed Riemann surface with at least one puncture.} We still have a branched cover $\pi: \Sigma \to C$ with $\Sigma\subset T_{\op{hol}}^*C$. Let $Z\subset C$ be the set of branch points, which are double branch points, and let $\widetilde Z\subset \Sigma$ be the set of ramification points. 

We define $\mathcal{B}\subset \Sigma$ to be a union of branch cuts given as follows: For each $p\in \widetilde Z$, choose two distinct arcs $b_p^+$ and $b_p^-$ that go from $p$ to punctures of $\Sigma$, so that:
\be
\item $\pi(b_p^+)=\pi(b_p^-)$ and goes from $\pi(p)\in Z$ to a puncture of $C$;
\item if $b_p^+$ (resp.\ $b_p^-$) lies on the graph of $df_i$ (resp.\ $df_j$) near $p$, then $\pi(b_p^+)$, pointing away from $\pi(p)$, is tangent to and directed by $\nabla(f_i-f_j)$ near $p$;
\item the arcs are oriented so that $b_p^+$ points away from $p$ and $b_p^-$ points towards $p$; 
\item the $b_p^+\cup b_p^-$ over all $p$ are disjoint; and
\item the $b_p^+\cup b_p^-$ are disjoint from $\pi^{-1}({\bf q})$.
\ee
Then we set 
$$\mathcal{B}:= \sqcup_{p\in \widetilde Z} (b_p^+\cup b_p^-).$$ 

We can refine the definition of $\mathcal{F}_{\op{hol}}$ in \eqref{eqn: Fhol} to take values in
$$\op{Mat}(N^{\kappa},\op{BSk}_\kappa(\Sigma))\otimes \Z[c^{\pm 1}],$$
so that $\mathcal{E}(u)$ is replaced by $\mathcal{E}(u) c^{\langle \mathcal{E}(u), \mathcal{B}\rangle}$, where $\langle \mathcal{E}(u), \mathcal{B}\rangle$ is the algebraic intersection number on $\Sigma$.  Referring to $\mathcal{T}:=\mathcal{T}^{\op{ind}=1}(\mathbf{p}_\alpha,\mathbf{y}',\mathbf{y},\mathbf{p}_\beta)$ from Section~\ref{subsection: proof of thm map of algebras holomorphic version}, if $(u_2,u_1)$ a boundary point of Type (T2) and $(w,u_3)$ is a boundary point of Type (T1) at opposite ends of a component of $\mathcal{T}$,  it is easy to verify that:
$$\langle \mathcal{E}(u_1)\mathcal{E}(u_2),\mathcal{B} \rangle= \langle \mathcal{E}(u_3),\mathcal{B} \rangle,$$
since $\mathcal{E}(u_1)\mathcal{E}(u_2)$ and $\mathcal{E}(u_3)$ are homotopic rel endpoints.

\begin{remark}
    Setting $c=q^{-1/2}$ gives factors analogous to those of \cite[Figure 18]{neitzke2020q}.
\end{remark}

\section{The hybrid Floer-Morse approach}
\label{section: hybrid Floer-Morse}

We discuss the hybrid Floer-Morse counterpart of the Floer-theoretic approach in the previous section. 

In \cite{yuan2023folded}, the third author defined a Morse theory on symmetric products of manifolds, where the $A_\infty$-structure is defined by counting {\em folded Morse flow trees}. We adapt the paper to our setting so that the modified version of folded Morse flow trees gives a quantization of the nonabelianization map.

\subsection{The base direction}

We will describe objects parametrizing the base direction, corresponding to the usual HDHF base directions $Y_{m,l}$ from Section~\ref{subsection: base direction}.

\begin{definition}
    A {\em folded metric ribbon tree} is a tree $T=(V(T),E(T))=(V,E)$ equipped with a length function
\begin{equation*}
    \ell\colon E\to \mathbb{R}_{\geq0}
\end{equation*}
such that $V=V_{\op{ext}}\sqcup V_{\op{int}}$ and the following hold:
\begin{enumerate}
    \item There exists a decomposition $V_{\op{ext}}=\{v_0\}\sqcup V_{\op{diag}}\sqcup V_{\op{cusp}}$, where $v_0$ is the unique root vertex. The vertices of $V_{\op{diag}}$ (resp.\ $V_{\op{cusp}}$) are called {\em diagonal} (resp.\ {\em cusp}) vertices. 
    \item Each $e\in E$ is directed by the direction pointing toward $v_0$.
    \item Each interior vertex $v\in V_{\op{int}}$ has valency $|v|\geq 3$ and there is a cyclic order $h_v\colon E_v\to\{0,1,\dots,|v|-1\}$, where $E_v$ is the set of edges adjacent to $v$, $h_v$ is a bijection, and $h_v=0$ for the outgoing edge.\footnote{The cyclic order is supposed to be the counterclockwise order around the vertex when the ribbon tree is embedded in the plane.}
\end{enumerate}
\end{definition}

Note that we are allowing $\ell(e)=0$.

Let $\mathcal{Y}_{m,k,r,s}$ be the space of pairs (informally called ``squids'')
\begin{equation*}
    Y_{m,k,r,s}\coloneqq (Y_{m,k}, \boldsymbol{T}:=(\{T_1,\dots,T_n\},\{t_1,\dots,t_n\})),
\end{equation*}
where $Y_{m,k}\in \mathcal{Y}_{m,k}$, $T_1,\dots,T_n$, $n\geq 0$, are folded metric ribbon trees, $0\leq t_1\leq\dots\leq t_n\leq1$,
$r=\sum_{i=1}^n|V(T_i)_{\op{diag}}|$, and $s=\sum_{i=1}^n |V(T_i)_{\op{cusp}}|$.
For each $T_i$, its root vertex is identified with $t_i\in [0,1]\simeq \partial_{m+1} Y_{m}$. Here we have fixed a smooth family of (orientation-reversing) parametrizations $\phi_{Y_m}\colon[0,1]\to\partial_{m+1} Y_{m}$ for all $Y_{m}\in\mathcal{Y}_m$.
See Figure \ref{fig: F-mix} for an example.  We will informally refer to $Y_{m,k}$ as the ``squid body'' and the $T_i$ as the ``squid legs''.\footnote{Strictly speaking, squids have eight ``arms'' and two ``tentacles''.}

\begin{convention}
    Sometimes we will omit certain subscripts of $Y_{m,k,r,s}$, e.g., write $Y_m$, $Y_{m,r}$, or $Y_{m,k,r}$, with the understanding that we are referring to the result of applying a forgetful functor. 
\end{convention} 

\begin{figure}[ht]
    \centering
    \includegraphics[width=8cm]{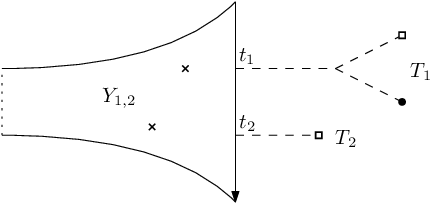}
    \caption{An example of $Y_{1,2,2,1}$. The crossings denote the marked points in $Y_{1,2}$. The boxes denote vertices in $V_{\op{diag}}$ and the black dot denotes a vertex in $V_{\op{cusp}}$.}
    \label{fig: F-mix}
\end{figure}

There is a naturally defined topology on $\mathcal{Y}_{m,k,r,s}$ admitting a compactification.

\subsection{Consistent collection of perturbation data}

We now describe a consistent collection of perturbation data associated to $\{\mathcal{Y}_{m,k,r,s}\}$.

Let $\mathfrak{M}(C)$ be the set of Riemannian metrics on $C$.
Given a Riemannian metric $g\in\mathfrak{M}(C)$, the induced {\em canonical almost complex structure} $J_g$ on $T^*C$ satisfies the following:
\begin{enumerate}
    \item $J_g$ is compatible with the canonical symplectic form $\omega$ on $T^*C$.
    \item $J_g$ maps vertical tangent vectors (i.e., those that are tangent to the fibers) to horizontal tangent vectors of $T^*C$ with respect to the Levi-Civita connection $\nabla_g$.
    \item On the zero section of $T^*C$, for $v\in T_q C\subset T_{(q,0)}(T^*C)$, let $J_g(v)(\cdot)=g(v,\cdot)\in T_q^*C\subset T_{(q,0)}(T^*C)$.
\end{enumerate}

Let $E_{\op{cusp}}\subset E$ be the subset of edges which are not on any directed path starting from $V_{\op{diag}}$ and ending on $v_0$. To each $Y_{m,k,r,s}\in\mathcal{Y}_{m,k,r,s}$, we can assign $Y_{m,k,r,0}\coloneqq \varphi(Y_{m,k,r,s})\in\mathcal{Y}_{m,k,r,0}$, obtained by collapsing all cusp edges $E_{\op{cusp}}$ of $Y_{m,k,r,s}$. We also denote the collapsing map by
\begin{equation*}
    \varphi\colon T_i\to \varphi(T_i)
\end{equation*}
for each $T_i$ of $Y_{m,k,r,s}$.

\begin{definition}
\label{def-perturb data}
    A {\em perturbation data for $\mathcal{Y}_{m,k,r,s}$} consists of:
    \begin{enumerate}
        \item a smooth family of Hamiltonian functions
        \begin{equation*}
            H_{Y_{m}}\colon Y_{m}\to \mathcal{H}(T^*C),
        \end{equation*}
        such that $H_{Y_m}$ takes values in $\mathcal{H}^\epsilon(T^*C)$ near $\partial_{m+1} Y_m$;
        \item a smooth family of almost complex structures
        \begin{equation*}
            J_{Y_{m}}\colon Y_{m}\to \mathcal{J}(T^*C,\omega);
        \end{equation*}
        \item for each $1\leq i\leq n$, a smooth family of Riemannian metrics
        \begin{equation*}
            g_{T_i}\colon T_i\to \mathfrak{M}(C);
        \end{equation*}
    \end{enumerate}
    subject to the following:
    \be
    \item[(4)] there exists $g\in\mathfrak{M}(C)$ such that $J_{Y_{m,k}}|_{\partial_{m+1}Y_m}=J_g$ and $g_{T_i}(v_{0})=g$ for each $1\leq i\leq n$;
    \item[(5)] writing $(E_i)_{\op{cusp}}$ for the edges of $T_i$ that do not admit an oriented path from a diagonal vertex, $g_{T_i}$ is constant on $(E_i)_{\op{cusp}}$;
    \item[(6)] $g_{T_i}$ is constant on a small neighborhood of each interior vertex and on each edge $e$ with $\ell(e)=0$.
    \ee
\end{definition}

\begin{remark} \label{rmk: take g to be almost constant on Ti}
    For all practical purposes, it suffices to take $g_{T_i}$ to be a small perturbation of the constant function $g$ from (4), where we take perturbations on small neighborhoods of a finite set of points of $T_i$, with one point from the interior of each edge. We want the first part of (6) to hold so that the Kirchhoff law from Definition~\ref{def-folded Morse tree}(FM4) is satisfied.
\end{remark}

\begin{definition}
\label{def-data with cusp}
    A {\em consistent collection of perturbation data} for $\{\mathcal{Y}_{m,k,r,s}\}$ consists of perturbation data for each $\mathcal{Y}_{m,k,r,s}$ which is compatible with compactification.  Moreover, we require that
    \begin{equation}
    \label{eq-collapse}
        g_{T_i}(y)=g_{\varphi(T_i)}(\varphi(y))
    \end{equation}
    for each $y\in T_i$ of $Y_{m,k,r,s}\in\mathcal{Y}_{m,k,r,s}$.
\end{definition}

\subsection{Hybrid moduli spaces}

The hybrid Floer-Morse approach involves defining a hybrid moduli space of pairs consisting of a holomorphic curve and a folded Morse tree (informally called a ``folded squid map'').

Given $\mathbf{y}\in CW(\sqcup_{i=1}^\kappa T^*_{q_i}C)$, let $\mathcal{T}^k_0(\mathbf{q},\mathbf{y},\mathbf{q})$ be the set of $((F,j),u)$, where $(F,j)$ is a compact Riemann surface with boundary, $\mathbf{w}_0,\mathbf{w}_1,\mathbf{w}_2$ are disjoint $\kappa$-tuples of boundary marked points of $F$, $\dot F=F\setminus\cup_i \mathbf{w}_i$, and
\begin{equation*}
    u\colon \dot{F}\to Y_1\times T^*C
\end{equation*}
is a continuous map such that
\begin{equation}
    (du-X_{H_{Y_1}}\otimes d\tau)^{0,1}_{j,J_{Y_{1,k}}}=0,
\end{equation}
and the following hold:
\begin{enumerate}
    \item For $m=0,1$, $\pi_{T^*C}\circ u(z)\in\sqcup_{i}T_{q_i}^*\Sigma$ if $\pi_{Y_1}\circ u(z)\subset\partial_m Y_1$. Each component of $\partial\dot F$ that projects to $\partial_m Y_1$ maps to a distinct $T^*_{q_i}C$.
    \item $\pi_{T^*C}\circ u(z)\in C$ if $\pi_{Y_1}\circ u(z)\subset\partial_2 Y_1$.
    \item $\pi_{T^*C}\circ u$ tends to $\mathbf{y}$, $\mathbf{q}$ as $s_1,s_2\to +\infty$ and tends to $\mathbf{q}$ as $s_0\to-\infty$.
    \item $\pi_{Y_1}\circ u$ is a $\kappa$-fold branched cover of $Y_1$, where we assume the branch points are simple and correspond to the marked points on $Y_{1,k}$.
\end{enumerate}

We then set
$$\mathcal{T}_0(\mathbf{q},\mathbf{y},\mathbf{q})=\sqcup_{k\geq0}\mathcal{T}^k_0(\mathbf{q},\mathbf{y},\mathbf{q}).$$

The following are analogs of Lemmas~\ref{lemma-dim} and \ref{lemma-cpt}, and are stated without proof.

\begin{lemma}
    \label{lemma-dim3}
    For a generic consistent collection of perturbation data, $\mathcal{T}_0(\mathbf{q},\mathbf{y},\mathbf{q})$ is of dimension $0$ and consists of discrete regular curves for all $\mathbf{y}$.
\end{lemma}

\begin{lemma} \label{lemma-cpt3}
    For a generic consistent collection of perturbation data and fixed $\mathbf{y}$ and $k\geq0$, the moduli space $\mathcal{T}^k_0(\mathbf{q},\mathbf{y},\mathbf{q})$ consists of finitely many curves.
\end{lemma}

\begin{definition}
    A {\em $\kappa$-folded lift} of a folded metric ribbon tree $T=(V,E)$ is a pair $(G,\pi_0)$ consisting of a metric ribbon graph $G=(\widetilde{V},\widetilde{E})$ and a map $\pi_0:G\to T$, constructed as follows: Let $\sqcup_\kappa T:=\sqcup_{i=1}^\kappa T^i$ be the disjoint union, where each $T^i=(V^i,E^i)$ is a copy of $T$. For each $v\in V$, denote the corresponding copies of vertices by $v^1,\dots,v^\kappa$.
    \be
    \item For each $v\in V_{\op{diag}}$, choose two vertices $v^i$ and $v^j$ and their adjacent edges $e^i$ and $e^j$; the other vertices are called {\em ordinary exterior vertices}. We then glue $e^i$ and $e^j$ as a ribbon graph by identifying $v^i$ and $v^j$. The resulting glued-up edge $e_{glue}$ has length $\ell(e_{glue})=\ell(e^i)+\ell(e^j)$. 
    \item For each $v\in V_{\op{cusp}}$, choose one vertex $v^i$ and call it a {\em cusp vertex}; the other vertices will also be called {\em ordinary exterior vertices}. 
    \ee
    The map $\pi_0$ is obtained from identifying each $T^i$ with $T$ and quotienting certain vertices of $\sqcup_\kappa T$ as described above; it restricts to a $\kappa$-fold covering map over $G\backslash \pi_0^{-1}(V_{\op{diag}})$.
\end{definition} 

\begin{definition} 
\label{def-folded Morse tree}
    A {\em folded Morse tree} is a pair $((G=(\widetilde{V},\widetilde{E}),\pi_0),\gamma)$, where $(G,\pi_0)$ is a $\kappa$-folded lift of a folded metric ribbon tree $T$ and $\gamma:G\to C$ is a continuous map satisfying the following:
    \begin{enumerate}
    \item[(FM1)]  For each contractible open subset $U\subset C\backslash Z$, there are $N$ smooth functions $f_1,\dots, f_N$ defined on $U$ such that the preimages $\pi^{-1}(U)$ are given by the graphs of $df_1,\dots,df_N$. (Recall the discussion from Section~\ref{section: holomorphic symplectic to real symplectic}.) For each $e\in\widetilde{E}$, when restricted to $(\gamma|_e)^{-1}(U)$,
    \begin{equation*}
        \gamma'(s)=\nabla_{g\circ \pi_0}(f_{i_{B}(e)}-f_{i_{T}(e)})(\gamma(s)),
    \end{equation*}
    for some $i_{B}(e), i_{T}(e)\in\{1,\dots,N\}$. (Here the subscripts $T$ and $B$ stand for ``top'' and ``bottom'', when the tree is drawn with the outgoing edge pointing to the left as in Figure~\ref{fig: F-mix} and the edges are oriented towards the root vertex.)

    \item[(FM2)] If $e\in \widetilde{E}$ is adjacent to an ordinary exterior vertex, then $i_{B}(e)=i_{T}(e)$ and $\gamma|_e$ is a constant map; otherwise $i_{B}(e)\not=i_{T}(e)$.

    \item[(FM3)] If $e\in \widetilde{E}$ is adjacent to a cusp vertex $v\in \widetilde{V}$, then the negative gradient flow of $f_{i_{B}(e)}-f_{i_{T}(e)}$ limits to a point in $Z$ when approaching $v$. 
    
    \item[(FM4)] The gradient flow preserves the ribbon structure, i.e., a Kirchhoff law: for each interior vertex $v\in \widetilde{V}$, denote the cyclically ordered edges adjacent to $v$ by $e_0,\dots,e_{|v|-1}$ (again in counterclockwise order), where $e_0$ is outgoing. Then
    \begin{align*}
        i_{B}(e_{i+1})&=i_{T}(e_i),\quad\text{for } 1\leq i\leq|v|-2,\\
        i_{T}(e_0)&=i_{T}(e_{|v|-1}),\quad i_{B}(e_0)=i_{B}(e_1).
    \end{align*}
\end{enumerate}
We will often omit $\pi_0$ from the notation for a folded Morse tree.
\end{definition}

Let $\mathbf{y}\in CW(\sqcup_{i=1}^\kappa T^*_{q_i}C)$ and $\alpha,\beta\in \mathfrak{I}$. 

\begin{definition} \label{defn: folded squid maps}
    A {\em folded squid map with base $Y_{1,k,r,s}$} is a pair $((F,u),(G,\gamma))$, where: 
\be
    \item[(FS1)] $Y_{1,k,r,s}\coloneqq (Y_{1,k}, \boldsymbol{T}:=(\{T_1,\dots,T_n\},\{t_1,\dots,t_n\}))\in \mathcal{Y}_{1,k,r,s}$;
    \item[(FS2)] $(F,u)\in\mathcal{T}^{k}_0(\mathbf{q},\mathbf{y},\mathbf{q})$ and $\dot F$ branch covers $Y_{1,k}$;
    \item[(FS3)] $G=\sqcup_{i=1}^n G_i$, $G_i$ is a $\kappa$-folded lift of $T_i$, and $(G_i,\gamma|_{G_i})$ is a folded Morse tree; 
    \item[(FS4)] for the root vertex $v_i\in T_i$, there is an identification 
    $$\eta_i: \pi_0^{-1}(v_i)\stackrel{\sim}\to(\pi_{Y_1}\circ u)^{-1}(\phi_{Y_1}(t_i))$$ 
    such that $\gamma(v)=\pi_{T^*C}\circ u(\eta_i(v))$ (recall that $\phi_{Y_1}$ is a map $[0,1]\to \bdry_2Y_1$);  
    \item[(FS5)] for consecutive $T_i$, $T_{i+1}$, there is an identification
    $$ \pi_0^{-1}(v_i)\stackrel{\eta_i}\to(\pi_{Y_1}\circ u)^{-1}(\phi_{Y_1}(t_i))\stackrel\sim\to (\pi_{Y_1}\circ u)^{-1}(\phi_{Y_1}(t_{i+1}))\stackrel{\eta_{i+1}^{-1}}\to \pi_0^{-1}(v_{i+1}),$$
    where the middle map is given by parallel transport along 
    $$(\pi_{Y_1}\circ u)^{-1}(\phi_{Y_1}([t_i,t_{i+1}])).$$
    If $v$ is taken to $v'$ under the identification and $e$ and $e'$ are edges adjacent to $v$ and $v'$, then the sheet corresponding to $i_{B}(e)$ agrees with the sheet corresponding to $i_{T}(e')$ under the parallel transport. 
\ee
\end{definition}

Let $\mathcal{P}_{k,r,s}(\mathbf{q},\mathbf{y},\mathbf{q})$ be the moduli space of folded squid maps $((F,u),(G,\gamma))$ with base $Y_{1,k,r,s}\in \mathcal{Y}_{1,k,r,s}$ and let
\begin{equation*}
    \mathcal{P}^{l}(\mathbf{q},\mathbf{y},\mathbf{q})\coloneqq\bigsqcup_{0\leq k\leq l, s\geq0}\mathcal{P}_{k, l-k, s}(\mathbf{q},\mathbf{y},\mathbf{q}).
\end{equation*}

One advantage of our setting compared to \cite{neitzke2020q} is that the related moduli space is of index $0$ for any number of loops in the graph:

\begin{lemma} \label{lemma: finite in number}
    For a generic consistent collection of perturbation data, $\mathcal{P}^{l}(\mathbf{q},\mathbf{y},\mathbf{q})$ consists of transversely cut out rigid elements such that:
    \be
    \item[(P1)] $Y_{1,k,l-k,s}$ satisfies $0 <t_1<\dots < t_n <1$;
    \item[(P2)] for each $T_i$, the length of the edge adjacent to a root vertex is positive and all the internal vertices are trivalent; 
    \item[(P3)] the marked points of $Y_{1,k}$ are disjoint from $\bdry_2 Y_1$; and
    \item[(P4)] the image of $\gamma$ is disjoint from the union $U_{R_\nu}$ of all $R_\nu$-disks in $C$ centered at points of $Z$ and $P$, with the exception of edges adjacent to cusp vertices that limit to cusp edges;
    \ee
    and which are finite in number for each $l\geq0$.
\end{lemma}

\begin{proof}
    Note that in Definition \ref{def-perturb data}, we allow $g_{T_i}$ to be domain-dependent on $E_i\backslash (E_i)_{\mathrm{cusp}}$ but constant on $(E_i)_{\mathrm{cusp}}$.
    By \cite[Lemma 2.6]{yuan2023folded}, the domain-dependent perturbation of $g_{T_i}$ on $E_i\backslash (E_i)_{\mathrm{cusp}}$ suffices to make the gradient trajectory along each edge adjacent to $V_{\mathrm{diag}}$ transversely intersect the big diagonal of $C^\kappa$.
    The perturbation of $g_{T_i}$ on the remaining part $(E_i)_{\mathrm{cusp}}$ is constant, but this is still sufficient for transversality using the inductive argument for flow trees starting from cusp vertices that is given in \cite[Proof of Theorem 1(i)]{casals2025spectralnetworksbettilagrangians}.
    Hence $\mathcal{P}^{l}(\mathbf{q},\mathbf{y},\mathbf{q})$ consists of transversely cut out rigid elements and the above discussion together with the transversality of the evaluation map 
    $$\mathcal{E}(u)=(\pi_{T^*C}\circ u) (\pi_{Y_1}\circ u)^{-1}(\bdry_2 Y_1)$$ 
    for $(u,\gamma)\in \mathcal{P}^{l}(\mathbf{q},\mathbf{y},\mathbf{q})$ implies (P1)--(P4). 
    
    By the combination of Gromov compactness and \cite[Lemma 2.7]{yuan2023folded}, $\mathcal{P}_{k, l-k, s}(\mathbf{q},\mathbf{y},\mathbf{q})$ is a finite set for each $l,k,s$. 
    Moreover, $\mathcal{P}_{k, l-k, s}(\mathbf{q},\mathbf{y},\mathbf{q})$ is empty for all but finitely many $s$ for each $l,k$ due to \cite[Proposition 3.18]{casals2025spectralnetworksbettilagrangians}. 
    Therefore, $\mathcal{P}^l(\mathbf{q},\mathbf{y},\mathbf{q})$ is a finite set for each $l\geq0$.
\end{proof}

\begin{figure}[ht]
    \centering
    \includegraphics[width=10cm]{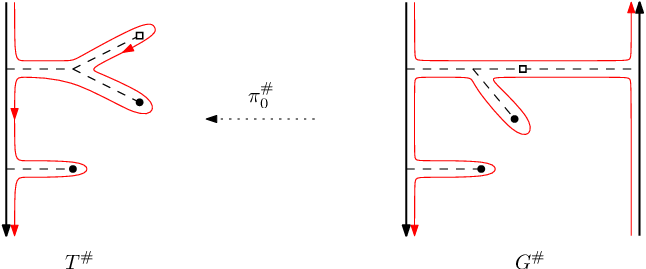}
    \caption{}
    \label{fig: folded tree}
\end{figure}

Now recall the parametrization $\phi_{Y_1}:[0,1]\to \bdry_2 Y_1$. Using this parametrization, the restriction of $\pi_{T^*C}\circ u$ to $(\pi_{Y_1}\circ u)^{-1}(\bdry_2 Y_1)$ gives a $\kappa$-strand braid $\mathcal{E}^\flat(u)$ on $C$ from ${\bf q}$ to ${\bf q}$; without loss of generality we may assume that $\mathcal{E}^\flat(u)$ is disjoint from the branch points of $C$.  

We define
$$Y_1^\sharp:= Y_1\sqcup N(T_1)\sqcup\dots \sqcup N(T_n)/\sim,$$ 
where $N(T_i)$, $i=1,\dots, n$, is an $\epsilon$-thick ribbon of $T_i$ for $\epsilon>0$ small, the root vertex $v_i$ of $T_i$ has an interval neighborhood $N(v_i)\simeq [-\epsilon,\epsilon]\times \{v_i\}\subset \bdry N(T_i)$, and $N(v_i)$ is glued to $\bdry_2 Y_1$ along $[t_i-\epsilon, t_i+\epsilon]$. Similarly, we define 
$$\dot F^\sharp:= \dot F\sqcup N(G_1)\sqcup \dots \sqcup N(G_n)/\sim,$$ 
where $N(G_i)$ is an $\epsilon$-thick ribbon of $G_i$ and $[-\epsilon,\epsilon]$ times the preimages of the root vertex of $T_i$ in $\bdry N(G_i)$ are glued to $\dot F$ using the identification $\eta_i$.

Let $T^\sharp$ be the boundary component of $Y_1^\sharp$ which nontrivially intersects $\bdry_2 Y_1$ and let $G^\sharp$ be the union of boundary components of $\dot F^\sharp$ that nontrivially intersect $(\pi_{Y_1}\circ u)^{-1}(\bdry_2 Y_1)$; $T^\sharp$ (resp.\ $G^\sharp$) is given by a red arc (resp.\ red arcs) in Figure \ref{fig: folded tree}. There is also a covering map $\pi_0^\sharp: G^\sharp\to T^\sharp$.

We parametrize $\phi_{Y_1}^\sharp: [0,1]\stackrel\sim\to T^\sharp$ (and hence each component of $G^\sharp$) so that $\phi_{Y_1}= \phi_{Y_1}^\sharp$ on the portion where $\bdry_2 Y_1$ and $T^\sharp$ agree. 
In view of (FS4) and (FS5), $\pi_{T^*C}\circ u|_{\bdry Y_1}$ and $\gamma$ can be consistently lifted to $\Sigma$, and gluing them after a small perturbation yields a $\kappa$-strand braid $G^\sharp\to \Sigma$, which contributes an entry $\mathcal{E}^\sharp(u,\gamma)$ of $\op{Mat}(N^{\kappa},\op{BSk}_\kappa(\Sigma))$.
We define
\begin{gather}
\label{eq-f-morse}
    \mathcal{F}_{\op{Mor}}\colon CW(\sqcup_{i=1}^\kappa T^*_{q_i}C)\to \op{Mat}(N^{\kappa},\op{BSk}_\kappa(\Sigma)),\\
    \nonumber\mathbf{y}\mapsto\sum_{l\geq0}\sum_{(u,\gamma)\in\mathcal{P}^l(\mathbf{q},\mathbf{y},\mathbf{q})} (-1)^{\sharp(u,\gamma)}\cdot\hbar^{\kappa-l}\cdot\mathcal{E}^\sharp(u,\gamma).
\end{gather}

To show that $\mathcal{F}_{\op{Mor}}$ is an algebra map, we need to analyze moduli spaces of dimension $1$.  However, a $1$-parameter family of spectral networks for $N>2$ may undergo a nonlocal topological change, and hence it is difficult to verify the invariance of $\mathcal{F}_{\op{Mor}}$ in this manner. Instead, we sketch the proof of $\mathcal{F}_{\op{Mor}}=\mathcal{F}_{\op{hol}}$ in the next section by comparing related moduli spaces, and the algebra map property is a byproduct of this equivalence.

\subsection{Modifications for $\Sigma^\epsilon_{\nu}$}

We describe the modifications when we replace $\Sigma$ by $\Sigma_{\nu}^{\epsilon}$. There exists a map
\begin{equation*}
    \mathcal{F}^{\nu,\epsilon}_{\op{Mor}}\colon CW(\sqcup_{i=1}^\kappa T^*_{q_i}C)\to \op{Mat}(N^{\kappa},\op{BSk}_\kappa(\Sigma^\epsilon_\nu)),
\end{equation*}
defined in the same way as \eqref{eq-f-morse}, except that the summation is over a slightly modified moduli space $\mathcal{P}_{\nu,\epsilon}^l(\mathbf{q},\mathbf{y},\mathbf{q})$ of folded squid maps with respect to $\Sigma^\epsilon_\nu$, where (FM3) of Definition \ref{def-folded Morse tree} is replaced by
\begin{itemize}
    \item[(FM3')] If $e\in \widetilde{E}$ is adjacent to a cusp vertex $v\in \widetilde{V}$, then the negative gradient flow of $f_{i_{B}(e)}-f_{i_{T}(e)}$ limits normally to a cusp edge in an $R_\nu$-disk centered at some point in $Z$ (see the red cusp edge in Figure~\ref{fig: cusp}). 
\end{itemize}

\begin{lemma}
\label{lemma-morse-compare}
    For each $l\geq0$, there exists $\nu_l>0$ such that
    $$(\zeta_\nu^{-1}\circ \pi_{\nu,\epsilon})_*\circ\mathcal{F}^{\nu,\epsilon}_{\op{Mor}}=\mathcal{F}_{\op{Mor}}\,\, \op{ mod } \hbar^{l+1},$$
    whenever $\nu> \nu_l$.
\end{lemma}

\begin{proof}
    For each $l\geq0$, $\mathcal{P}^l(\mathbf{q},\mathbf{y},\mathbf{q})$ is a finite set by Lemma~\ref{lemma: finite in number}. 
    By \cite[Lemma 3.23]{casals2025spectralnetworksbettilagrangians}, for $\nu_l>0$ large, there is a bijection between $\mathcal{P}^{l'}(\mathbf{q},\mathbf{y},\mathbf{q})$ and $\mathcal{P}_{\nu,\epsilon}^{l'}(\mathbf{q},\mathbf{y},\mathbf{q})$ for all $\nu>\nu_l$ and $l'\leq l$. 
    Moreover, the contribution of each element in $\mathcal{P}_{\nu,\epsilon}^{l'}(\mathbf{q},\mathbf{y},\mathbf{q})$ to $(\zeta_\nu^{-1}\circ \pi_{\nu,\epsilon})_*\circ\mathcal{F}^{\nu,\epsilon}_{\op{Mor}}$ is the same as the contribution of its corresponding element in $\mathcal{P}^{l'}(\mathbf{q},\mathbf{y},\mathbf{q})$ to $\mathcal{F}_{\op{Mor}}$. Hence $(\zeta_\nu^{-1}\circ \pi_{\nu,\epsilon})_*\circ\mathcal{F}^{\nu,\epsilon}_{\op{Mor}}=\mathcal{F}_{\op{Mor}}$ mod $\hbar^{l+1}$ for $\nu>\nu_l$.
\end{proof}

\section{The equivalence of the two approaches}
\label{section: equivalence of two approaches}

In this section we sketch the proof of the required lemmas which imply Conjecture \ref{conj-morse}. 

\subsection{From squids to Riemann surfaces}

In this subsection we define a gluing map $\Theta_\lambda$, for $\lambda>0$ small, from the space of squids to the space of holomorphic disks with interior marked points and present its basic properties. This is a generalization of Theorem 10.4 and Section 14 of \cite{fukaya1997zero} to the case of disks with interior marked points. 

\subsubsection{Thickening}  \label{subsubsection: thickening}

Let  $T=(V,E)$  be a folded metric ribbon tree.  For each edge $e\in E$ we define
\begin{gather*}
    \Theta_\lambda(e)\coloneqq \{z\in \mathbb{C}\,|\,0\leq\op{Re}z\leq \ell(e),-\lambda/2\leq\op{Im}z\leq \lambda/2\}\backslash \{0,\ell(e)\},
\end{gather*}
which is a closed strip with two boundary punctures. We view $\Theta_\lambda(e)$ as a thickening of $\op{int}(e)= \{z\in \mathbb{C}\,|\,0\leq\op{Re}z\leq \ell(e),\op{Im}z=0\}$, where $e$ is oriented in the {\em negative} $\op{Re}(z)$-direction. The two boundary punctures split the two vertical boundary components into four pieces.

Next we define the {\em thickening $\Theta_\lambda(T)$} of $T$ as follows:
For an interior vertex $v\in V_{\op{int}}$, denote the cyclically ordered edges adjacent to $v$ by $e_0,\dots,e_{|v|-1}$ such that $h_v(e_i)=i$, where $e_0$ is the outgoing edge. We then identify the bottom left vertical boundary of $\Theta_\lambda(e_i)$ with the top left vertical boundary of $\Theta_\lambda(e_{i+1})$ for $i=1,\dots,|v|-1$; identify the bottom left vertical boundary of $\Theta_\lambda(e_{|v|-1})$ with the bottom right vertical boundary of $\Theta_\lambda(e_0)$; identify the top left boundary of $\Theta_\lambda(e_{1})$ with the top right vertical boundary of $\Theta_\lambda(e_0)$.
We then fill the interior removable puncture.
See Figure \ref{fig: inner glue}.
\begin{figure}[ht]
    \centering
    \includegraphics[width=11cm]{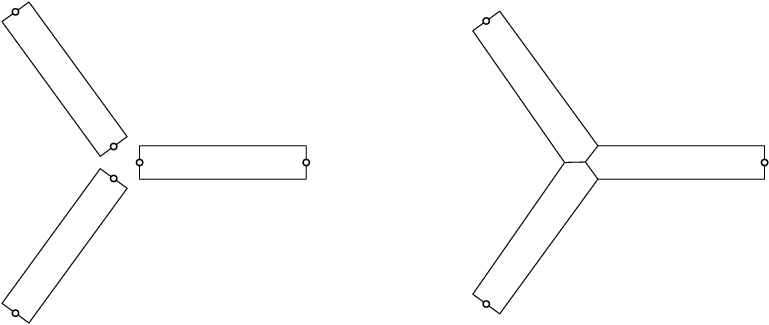}
    \caption{}
    \label{fig: inner glue}
\end{figure}

If $e$ is adjacent to a diagonal vertex $v\in V_{\op{diag}}$, we identify the top right and bottom right vertical boundaries of $\Theta_\lambda(e)$. 
We then fill the removable puncture and convert it into an interior marked point.
See Figure \ref{fig: box glue}.
\begin{figure}[ht]
    \centering
    \includegraphics[width=10cm]{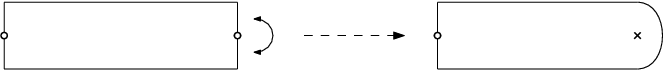}
    \caption{}
    \label{fig: box glue}
\end{figure}

If $e$ is adjacent to a cusp vertex $v\in V_{\op{cusp}}$, we identify the top right and bottom right vertical boundaries of $\Theta_\lambda(e)$, fill the removable puncture but do not convert it into a marked point.

The result of the above operations will be denoted $\Theta_\lambda(T)$.

\begin{remark}\label{rmk: when ell is zero}
    Observe that if $\ell(e)=0$ for an edge $e$ adjacent to a diagonal vertex, then $\Theta_\lambda(T)$ is equivalent to $\Theta_\lambda(T')$, where $T'$ is obtained by removing the edge $e$ and concatenating the adjacent edges $e',e''$.  There will also be an interior marked point at $(\ell(e'),0)\in \Theta_\lambda(e'\cup e'')$.
\end{remark}

\subsubsection{Gluing squid legs} \label{subsubsection: gluing squid legs}

Fix $\lambda>0$ small. Let $l\geq k\geq 0$. Let $\mathcal{Y}^\lambda_{m,k,l-k,s}$ be the subset of $\mathcal{Y}_{m,k,l-k,s}$ consisting of 
$$Y_{m,k,l-k,s}=(Y_{m,k}, (\{T_1,\dots, T_n\}, \{t_1,\dots,t_n\})),$$ 
where the following spacing condition holds:
\begin{itemize}
    \item[(Sp)] $t_0:=0< t_1<\dots <t_n< 1=:t_{n+1}$ and each $t_{i+1}-t_i> 2\lambda$.
\end{itemize}
(Sp) is a more precise version of (P1) from Lemma~\ref{lemma: finite in number}. Also, for $C\gg 0$ fixed, let $\mathcal{Y}^{\lambda,C}_{m,k,l-k,s}$ be the subset of $\mathcal{Y}^\lambda_{m,k,l-k,s}$ such that none of the trees $T_i$ have edges adjacent to a diagonal vertex with $\ell< C\lambda$.

For the purposes of analyzing $\mathcal{F}^{\nu,\epsilon}_{\op{Mor}}$ it suffices to consider $m=1$. 

Given $Y_{1,k,l-k,s}\in \mathcal{Y}^\lambda_{1,k,l-k,s}$, we glue the thickening $\Theta_\lambda(T_i)$ of each tree $T_i$ to $Y_{1,k}$ as follows: Let $e_0$ be the edge adjacent to the root vertex of $T_i$. Then we glue $\{\op{Re}z=0\}\cap \Theta_\lambda(e_0)\subset \Theta_\lambda(T_i)$ to $\partial_{2}Y_{1}$ via the map
\begin{gather*}
    z\mapsto \phi_{Y_1}(t_i-\op{Im}z),
\end{gather*}
and then fill the removable puncture. 
See Figure \ref{fig: root glue}.
\begin{figure}[ht]
    \centering
    \includegraphics[width=13cm]{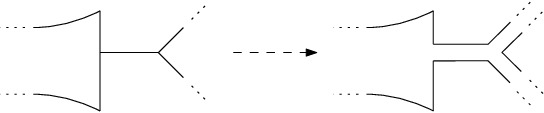}
    \caption{}
    \label{fig: root glue}
\end{figure}

For each $l\geq0$, we have defined the map
\begin{equation}
    \Theta_\lambda\colon \bigsqcup_{0\leq k\leq l,\,s\geq0}\mathcal{Y}^\lambda_{1,k,l-k,s}\to \mathcal{Y}_{1,l}.
\end{equation}

The following lemma is a generalization of \cite[Lemma 14.5]{fukaya1997zero} and is stated without proof: 

\begin{lemma}
\label{lemma-local diff}
    Suppose $C>0$ is a fixed large constant, $\lambda_0>0$ is sufficiently small, and $\lambda\in(0,\lambda_0]$.  Then
    $\Theta_\lambda|_{s=0}$ is a local diffeomorphism, when restricted to the subset 
    $$\bigsqcup_{0\leq k\leq l} \mathcal{Y}^{\lambda,C}_{1,k,l-k,0} \subset \bigsqcup_{0\leq k\leq l} \mathcal{Y}^\lambda_{1,k,l-k,0}.$$ 
\end{lemma}

\subsubsection{Additional requirements for a consistent choice of perturbation data for $\{\mathcal{Y}_{1,l}\}$} \label{subsubsection: additional requirements}

Given a consistent collection of perturbation data on $\{\mathcal{Y}_{1,k,r,0},\,k,r\geq0\}$ and $\lambda_0>0$ small, our goal is to impose additional conditions on a consistent choice of perturbation data for $\{\mathcal{Y}_{1,l}\}$, which depends on $\Sigma^{\epsilon}_\nu$ (i.e., on $\epsilon$ and $\nu$) and $\lambda\in (0,\lambda_0]$, so that it is compatible with the perturbation data on $\{\mathcal{Y}_{1,k,r,0},\,k,r\geq0\}$. For consistency, we want $J_{Y_{1,l}}$ for $\Sigma^\epsilon_\nu$ to limit to $J_{Y_1}$ as $\epsilon\to 0$, $\nu\to \infty$, and $\lambda\to 0$. We do not need to consider the Hamiltonian functions because they are assumed to be zero near the zero section of $T^*C$.

The main complication is that, in contrast to \cite{fukaya1997zero}, $\Theta_\lambda|_{s=0}$ is neither injective nor surjective. Given a sequence $({Y}^n_{1,k+r})_{n\in \Z_{>0}}$ in $\mathcal{Y}_{1,k+r}$ that limits to $\mathcal{Y}_{1,k,r,0}$, the key issue is to distinguish which branched points limit to diagonal vertices and which ones remain as branched points.  Since $\Theta_\lambda|_{s=0}$ is not injective, there is an ambiguity in distinguishing these two groups. 
The solution is to specify $J_{Y_{1,k+r},\lambda}$ only on a subset of $\mathcal{Y}_{1,k+r}$ where we can easily separate the two groups of branched points, and then extend it to all of $\mathcal{Y}_{1,k+r}$ in a smooth manner.

We first observe that all the marked points of $\Theta_\lambda(Y_{1,0,r,0})$ that arise from the diagonal vertices are close to $\partial_2Y_1$.  Let us use the model $(-\infty,0]\times [0,1]\subset \C$ for $Y_1$ from Equation~\eqref{eqn: model for Ym-1}. Then:

\begin{lemma}
\label{lemma-position}
    Fix a positive integer $l$ and $\alpha_0>0$ small. There exist $\lambda_0>0$ small and a smooth function $g:(0,\lambda_0)\to \R_{>0}$ such that:
    \be
    \item $\lim_{\lambda\to 0} g(\lambda)=0$;
    \item for any $\lambda\in (0,\lambda_0)$ and $Y_{1,0,r,0}\in\mathcal{Y}^\lambda_{1,0,r,0}$, $0\leq r\leq l$, all the marked points in $\Theta_\lambda(Y_{1,0,r,0})$ satisfy $\op{Re}z>-g(\lambda)$; and
    \item the restriction of $\Theta_\lambda$ to $\mathcal{Y}^{\lambda,C}_{1,0,r,0}$ is a bijection onto its image in $\mathcal{Y}_{1,r}$ for $\lambda\in (0,\lambda_0)$.
    \ee
\end{lemma}

Letting $h:(0,\lambda_0)\to \R_{>0}$ be a smooth function such that $\lim_{\lambda\to 0} h(\lambda)=0$ and $\lim_{\lambda\to 0} \tfrac{h(\lambda)}{g(\lambda)}=+\infty$, we define
\begin{equation*}
    E_{k,r}^\lambda:=\{Y_{1,k+r}\in\mathcal{Y}_{1,k+r}\,|\,Y_{1,k+r}\in\Theta_\lambda(\mathcal{Y}^\lambda_{1,k,r,0}),\,\text{$k$ marked points have $\op{Re}z<-h(\lambda)$}\},
\end{equation*}
for all $k+r=l$ and $0<\lambda<\lambda_0$. In words, $Y_{1,k+r}\in E_{k,r}^\lambda$ has $k$ marked points with $\op{Re}z<-h(\lambda)$ and the remaining $r$ marked points coming from the diagonal vertices satisfy $\op{Re}z>-g(\lambda)$.  We also define $E_{k,r}^{\lambda,C}$ to be the subset of $E_{k,r}^\lambda$ such that $Y_{1,k+r}$ is in the $\Theta_\lambda$-image of $\mathcal{Y}^{\lambda,C}_{1,k,r,0}$.

\begin{figure}[ht]
    \centering
    \includegraphics[width=8cm]{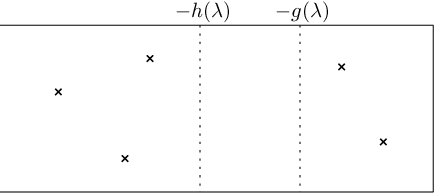}
    \caption{An example of $E^\lambda_{k,r}$.}
    \label{fig: branch lambda}
\end{figure}

We are now in a position to define the family $J_{Y_{1,k+r},\lambda}$ of almost complex structures. Given $Y_{1,k+r}=\Theta_\lambda(Y_{1,k,r,0})\in E^{\lambda}_{k,r}$, a tree $T_i$ of $Y_{1,k,r,0}$, and an edge $e\subset T_i$:
\be
\item if $Y_{1,k+r}\in E^{\lambda,C}_{k,r}$, then we set:
\begin{equation} \label{eqn: Jan}
    J_{Y_{1,k+r},\lambda}(w)\coloneqq J_{g_{T_i}(\op{Re}w)} ~~ \mbox{ on } ~~\Theta_\lambda(e)\subset\mathbb{C}_w,
\end{equation}
where we are identifying $e$ with $[0,\ell(e)]$ (the latter is oriented in the negative direction) and $g_{T_i}$ is as given in Definition~\ref{def-perturb data};
\item if $Y_{1,k+r}\in E^{\lambda}_{k,r}\setminus E^{\lambda,C}_{k,r}$, then for all edges $e$ adjacent to a diagonal vertex with $\ell< C\lambda$:
\begin{equation}
    J_{Y_{1,k+r},\lambda}(w)\coloneqq J_{g_{T_i}(0)} ~~ \mbox{ on } ~~\Theta_\lambda(e)\subset\mathbb{C}_w,
\end{equation}
and for all other edges we use \eqref{eqn: Jan}.
\ee 
We also set $J_{Y_{1,k+r},\lambda}=J_{Y_{1}}$ on $\op{Re}z<-h(\lambda)$. Since $\mathcal{J}(T^*C,\omega)$ is contractible, we can easily extend $J_{Y_{1,k+r},\lambda}$ to all $Y_{1,k+r}\in \mathcal{Y}_{1,k+r}$ and all $0<\lambda\leq \lambda_0$ in a smooth manner. We can also take $J_{Y_{1,k+r},\lambda}$ to be compatible with the compactification.

\subsection{The correspondence between two moduli spaces} \label{subsection: the correspondence}

The goal of Section~\ref{subsection: the correspondence} is to give a sketch of the proof of Conjecture \ref{conj-morse}, which is given in four steps corresponding to the following four subsections. We remark that all the steps of the proof have similar/analogous results that have already appeared in the literature.

\subsubsection{Step 1: The limit as $\epsilon\to 0$} \label{subsubsection: compactness}

In this step we take a subsequential limit of holomorphic maps $v\in \mathcal{T}^{l}_{\nu, \epsilon}(\mathbf{p}_\alpha,\mathbf{y},\mathbf{p}_\beta)$ as $\epsilon\to 0$ (for $\nu\gg 0$ fixed) and obtain a folded squid map. (Recall that the subscripts $\nu,\epsilon$ in $\mathcal{T}^{l}_{\nu, \epsilon}$ indicate that the target of the evaluation map is $\Sigma^{\epsilon}_\nu$ and $l$ indicates the number of marked/branch points.)

The precise statement is the following compactness result, which follows from Cant-Chen~\cite{CantChen2023adiabatic} (which in turn is a generalization of Floer \cite{floer1988morse}, Fukaya-Oh \cite{fukaya1996morse}, and Ekholm~\cite[Section 5]{ekholm2007morse}) and a more careful analysis of the limit gradient graph.

\begin{theorem}
    \label{thm: Morse-Floer-cpt}
    Let $l_0\geq 0$, $0\leq  l\leq l_0$, $\nu>0$ sufficiently large, and $\lambda>0$ sufficiently small.  Given a sequence
    \begin{equation*}
        (v_j: \dot F_j\to Y^j_{1,l}\times T^*C )\in\mathcal{T}^{l}_{\nu, \epsilon_j}(\mathbf{p}_\alpha,\mathbf{y},\mathbf{p}_\beta),
    \end{equation*}
    where $\epsilon_j\to 0$, after passing to a subsequence while keeping the same notation, $v_j$ limits to a folded squid map $(u,\gamma)\in\mathcal{P}^\nu_{k,l-k}(\mathbf{q},\mathbf{y},\mathbf{q})$ for some $0\leq k\leq l$, with domain 
    $$Y_{1,k,l-k,s}=(Y_{1,k}, (\{T_1,\dots,T_n\},\{t_1,\dots, t_n\})),$$ 
    such that, assuming the genericity of the perturbation data, (Sp) and (P2), (P4) from Lemma~\ref{lemma: finite in number}, as well as (P3') below (an upgrade of (P3)) hold.
    \be
    \item[(P3')] $Y_{1,k,l-k}\in E_{k,l-k}^\lambda$.
    \ee
\end{theorem}

Given locally defined real-valued functions $f_0, f_1$ on $C$ and a gradient trajectory $\gamma(s)$ of $\nabla_{g(\gamma(s))} (f_1-f_0)$, i.e., a map $\gamma:[A,B]_s\to C$ such that $\gamma'(s)=\nabla_{g(\gamma(s))} (f_1-f_0)(\gamma(s))$, an {\em approximate holomorphic strip corresponding to $\gamma$} is the map 
\begin{gather*}
    w_\gamma:[A,B]_s\times [0,1]_t\to T^*C,\\
    (s,t)\mapsto (\gamma(s), t df_1(\gamma(s))+ (1-t)df_0(\gamma(s)).
\end{gather*}
Note that $\op{Im}w_\gamma|_{t=0}\subset \op{Graph}(df_0)$ and $\op{Im} w_\gamma|_{t=1}\subset \op{Graph}(df_1)$, and we say $w_\gamma$ is {\em from $df_0$ to $df_1$.}

By ``$v_j$ limiting to $(u,\gamma)$'' we mean that, for $j\gg 0$, $v_j$ is close to breaking into $u$ and a collection of holomorphic strips that are close to approximate holomorphic strips corresponding to $\gamma|_e$, where $e$ ranges over all the edges of $G$ and the functions are defined by $\Sigma^{\epsilon_j}_\nu$ and $\gamma|_e$.

\begin{proof}
    By \cite{CantChen2023adiabatic} and \cite[Section 5]{ekholm2007morse}, after passing to a subsequence, the holomorphic map 
    $$v_j: \dot F_j\to Y^j_{1,l}\times T^*C$$ 
    limits to some pair $(u,\gamma)$ consisting of $u\in  \overline{\mathcal{T}_{\nu,0}^l({\bf q}, {\bf y}, {\bf q})}$ and an $n$-tuple 
    $$\gamma=(\gamma_1: G_1\to C,\dots, \gamma_n: G_n\to C)$$ 
    of Morse gradient graphs.  We first observe that ghost bubbles (both interior and boundary) can be eliminated using \cite{DoanWalpuski2023} and \cite{EkholmShende2022Ghost}; also see \cite[Section 3.7]{colin2020applications}.
    
    We claim that, for each $i=1,\dots, n$, the domain $G_i$ of $\gamma_i$ is a $\kappa$-folded lift $\pi_0: G_i\to T_i$ of a folded metric ribbon tree $T_i$.  Since each $\dot F_j$ is a $\kappa$-fold branched cover of $Y^j_{1,l}\in \op{Im}\Theta_\lambda$, there exists $\pi_0: G_i\to T_i$ which is a $\kappa$-fold cover except over (i) the diagonal vertices and (ii) finitely many interior non-vertex points $b_1,\dots, b_m$ of $T_i$, such that $\# \pi_0^{-1}(b_{i'})< \kappa$. Each vertex of type (i) corresponds to a branch point that is pushed to the ends and each vertex of type (ii) corresponds to a cluster of branch points are close together and remain in the interior.  Our goal is to eliminate the possibility of (ii). Since $\mathcal{E}(u)$ is generic and is a braid when viewed in $[0,1]_t\times C$, each constant $t$ slice of $\mathcal{E}(u)$ is a disjoint $\kappa$-tuple of points on $C$ and $\gamma_i$ is injective near its terminal points (i.e., $\pi^{-1}_0$ of the root vertex of $T_i$).  Suppose there are two edges $e_1$ and $e_2$ of $G_i$ that start at the same vertex over $b_{i'}$ and end over the root vertex. Since $\ell(e_1)=\ell(e_2)$, the gradient trajectories $\gamma_i(e_1)$ and $\gamma_i(e_2)$ must meet {\em synchronously} at their initial point. The perturbation data requirements from Section~\ref{subsubsection: additional requirements} make the gradient trajectories generic, and the synchronous meeting imposes a codimension $1$ condition on $t$.  On the other hand, the gradient trajectories emanating from cusp edges and the gradient trajectories $e_{glue}$ corresponding to diagonal vertices are rigid, and the existence of a type (ii) vertex overdetermines $t$.  Therefore, type (ii) vertices do not exist and the claim follows.
    
    Since the virtual dimension of $\mathcal{T}^{l}_{\nu, \epsilon_j}(\mathbf{p}_\alpha,\mathbf{y},\mathbf{p}_\beta)$ is $0$, there cannot be more than one level in $u\in \overline{\mathcal{T}_{\nu,0}^l({\bf q}, {\bf y}, {\bf q})}$ and hence $u\in \mathcal{T}_{\nu, 0}^l({\bf q}, {\bf y}, {\bf q})$. 
    The conditions (Sp), (P2), (P3'), and (P4) follow from the genericity of the perturbation data and Lemma~\ref{lemma: finite in number}.
\end{proof}

\subsubsection{Step 2: Adjusting the boundary of the squid body}

Let 
$$(u:\dot F\to Y_1\times T^*C,\gamma)\in \mathcal{P}^\nu_{k,l-k}({\bf q}, {\bf y}, {\bf q}),$$
which we assume satisfies (P1)--(P4) by Lemma~\ref{lemma: finite in number}. By definition, $\pi_{T^*C}\circ u$ maps the portion of $\bdry \dot F$ over $\bdry_2 Y_1$ (i.e., $(\pi_{Y_1}\circ u)^{-1}(\bdry_2 Y_1)$) to $C$. 

The goal of this step is to modify $u$ to $u^\circ$ (aka ``adjust the boundary of the squid body'') so that $(\pi_{Y_1}\circ u^\circ)^{-1}(\bdry_2 Y_1)$ is mapped to $\Sigma^\epsilon_\nu$, except for small neighborhoods of points where the squid legs are to be attached to $u^\circ$.

Let us view $\mathcal{E}(u)$ as a braid in $[0,1]_t\times C$.
Using (FS4) and (FS5) of Definition~\ref{defn: folded squid maps}, we can obtain a piecewise continuous lift $\mathcal{E}_{\epsilon,\nu}(u)$ of $\mathcal{E}(u)$, viewed as braid in $[0,1]_t\times\Sigma^\epsilon_\nu$ such that:
\be
\item[(B1)] $\mathcal{E}_{\epsilon,\nu}(u)$ starts at some ${\bf p}_\alpha$ and ends at some ${\bf p}_\beta$;
\item[(B2)] at each vertex $v\in \sqcup_i\pi^{-1}_0(v^i)$, where $v^i$ is the root vertex of $T_i$, as $t$ increases along $\bdry_2 Y_1$, the lift $\mathcal{E}_{\epsilon,\nu}(u)$ jumps from the sheet $\op{Graph} df_{i_T(v)}$ to the sheet $\op{Graph} df_{i_B(v)}$; and
\item[(B3)] these discontinuities at $t=t_i$ are the only discontinuities.
\ee

Suppose $\mu>0$ is small. Consider squares 
$$S_{v,\mu}=[-\mu/2,\mu/2]_x\times[-\mu/2,\mu/2]_y\subset C$$ 
of width $\mu$, centered $\sqcup_i \mathcal{E}(u)|_{t=t_i}$ and indexed by $v\in  \sqcup_i\pi^{-1}_0(v^i)$, so that $\mathcal{E}(u)$ intersects each square along $y=0$ and is oriented in the positive $x$-direction. If $\epsilon=\epsilon(\mu)>0$ is sufficiently small, say in $O(\mu^2)$, then we can modify $\Sigma^\epsilon_\nu$ to $\widetilde\Sigma^\epsilon_\nu$ over each $S_{v,\mu}$ as follows: 
\be
\item remove the two sheets of the preimage $\pi^{-1}(S_{v,\mu})$ corresponding to $df_{i_T(v)}$ and $df_{i_B(v)}$, where $\pi:\Sigma^\epsilon_\nu\to C$ is the projection; and
\item replace it with one sheet $d\tilde f$, where $\tilde f$ interpolates between $f_{i_T(v)}$ along $x=-\mu/2$ and $f_{i_B(v)}$ along $x=\mu/2$ and has small derivative.
\ee

\begin{figure}[ht]
    \centering
    \includegraphics[width=14cm]{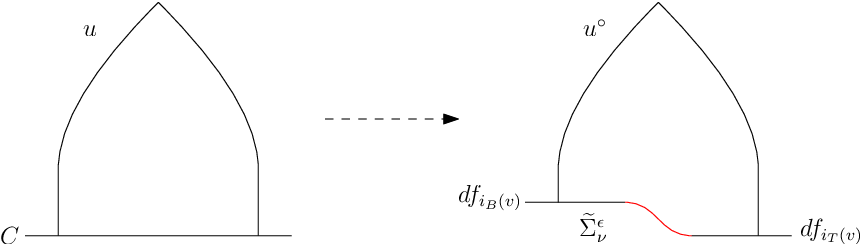}
    \caption{}
    \label{fig: picture 1}
\end{figure}

\begin{lemma} \label{lemma: existence of ucirc}
    For $\mu>0$ small and $\epsilon=\epsilon(\mu)>0$ small such that $\epsilon\in O(\mu^2)$, there exists a unique holomorphic map $u^\circ: \dot F\to Y_1\times T^*C$ in the moduli space 
    $\mathcal{T}^k_{\nu,\widetilde \Sigma ^\epsilon_\nu} ({\bf p}_\alpha,{\bf y}, {\bf q}_\beta)$, which is defined in the same way as $\mathcal{T}^k_{\nu,0}({\bf q}, {\bf y}, {\bf q})$, except that: 
    \be
    \item $({\bf q},{\bf y},{\bf q})$ is replaced by $({\bf p}_\alpha, {\bf y}, {\bf p}_\beta)$, where ${\bf p}_\alpha$ and ${\bf p}_\beta$ are as given in (B1); and 
    \item $C$ is replaced by $\widetilde \Sigma ^\epsilon_\nu$,
    \ee
    and which is close to $u$. Moreover, $u^\circ$ is transversely cut out.
\end{lemma}

\begin{proof}
    This is immediate from the regularity of $u\in \mathcal{T}^k_{\nu,0}({\bf q}, {\bf y}, {\bf q})$ and Gromov compactness.  Note that $\widetilde \Sigma^\epsilon_\nu$ has boundary, but if $\epsilon>0$ is small in response to $\mu$, the curve $u^\circ$ will stay away from the boundary.
\end{proof}

\subsubsection{Step 3: Gluing Morse gradient graphs}

In this step we construct the holomorphic curves ``corresponding'' to Morse gradient graphs. 

Let $M$ be a $n$-dimensional manifold, $g$ a Riemannian metric on $M$, and $f_0,f_1$ generic smooth functions on $M$ (e.g., such that $f_1-f_0$ is Morse). Let $\gamma:[A,B]_\sigma\to M$ be a trajectory of $\nabla_{g(\sigma)} (f_1-f_0)$ which avoids a critical point of $f_1-f_0$, where $\nabla_{g(\sigma)}$ is with respect to a family $g(\sigma)$ of metrics which is $C^1$-close to $g$. For $\epsilon>0$ small, we define the {\em rescaled trajectory} $\gamma_\epsilon:[A/\epsilon,B/\epsilon]_s\to C$ as the trajectory of $\nabla_{g(\epsilon s)} \epsilon(f_1-f_0)$ with the same image as $\gamma$.

\begin{definition}
    For $\epsilon>0$ small, a $J_{g(\epsilon s)}$-holomorphic strip 
    $$v:[A/\epsilon,B/\epsilon]_s\times [0,1]_t\to T^*M$$ 
    from $\epsilon df_0$ to $\epsilon df_1$ {\em approximates $\gamma$} if the following hold:
    \be
    \item $J_{g(\epsilon s)}$ is a domain-dependent canonical almost complex structure that is close to $J_g$; 
    \item $v$ is $C^0$-close to an approximate holomorphic strip corresponding to $\gamma_\epsilon$ from $\epsilon df_0$ to $\epsilon df_1$, where the error is of order $O(\epsilon)$.  
    \ee
\end{definition}

\begin{definition}
    For $\epsilon>0$ small, a holomorphic map $v: \Theta_\lambda(T)\to T^*M$ (with domain-dependent $J$) {\em approximates} a Morse gradient graph $\gamma:T\to M$ (with Morse functions depending on the edges), if $v|_{\Theta_\lambda(e)}$ approximates $\gamma(e)$ for each edge $e$ of $T$, 
\end{definition}

\n
{\em Virtual extensions of holomorphic strips.} Let $\epsilon>0$ small.  Let $f_0,f_1:M\to \R$ be smooth functions that are linear on an open neighborhood $U\subset \R^n\subset M$ of $0$, and let $\gamma:(-\infty,0]_s\to \R^n$ be a trajectory of $\nabla (f_1-f_0)$ emanating from a critical point of $f_1-f_0$, such that $\gamma(0)=0$.  Assume that $g$ and $J_g$ are standard on $U\subset \R^n$ and $T^*U\subset \C^n$. Observe that there exists $C> 0$ for which $\gamma|_{[-C,0]}$ has image in $U$ and that $df_0, df_1$ are constant on $U$.   Let 
$$v:(-\infty,0]_s\times[0,1]_t\to T^*M$$ 
be a holomorphic half-strip from $\epsilon df_0$ to $\epsilon df_1$ that approximates $\gamma$.  

\begin{definition} \label{defn: virtual extension} $\mbox{}$
\be
\item A {\em virtual extension $\widetilde{v}$} of $v$, if it exists, satisfies the following:
\begin{itemize}
\item $\widetilde{v}=v$ on $(-\infty,0]\times[0,1]$; and
\item $\widetilde{v}$ on $[-C,\infty)\times [0,1]$ is the extension of $v|_{[-C,0]\times[0,1]}$ to a holomorphic half-strip $[-C,\infty)\times[0,1]\to \C^n$ from $\epsilon d\widetilde f_0$ to $\epsilon d\widetilde f_1$ that approximates a gradient trajectory $\widetilde \gamma$ of $\widetilde  f_1-\widetilde f_0$, where $\widetilde f_i: \R^n\to \R$ is a smooth extension of $f_i|_U$ such that $\widetilde \gamma$ ends at a critical point of $\widetilde f_1-\widetilde f_0$.
\end{itemize}
\item  A holomorphic half strip $v$ is {\em virtually extendable} if a virtual extension $\widetilde{v}$ of $v$ exists.
\ee
\end{definition}

We can similarly define the notions of virtual extension and virtual extendability to holomorphic strips $v: [A,B]_s\times[0,1]_t\to T^*M$ and to holomorphic curves that approximate Morse gradient graphs.

\s\n
{\em Edge models.}
We first prepare {\em edge models} of the following type: holomorphic strips corresponding to (a) finite-length gradient trajectories, (b) gradient trajectories that end at cusp edges, and (c) gradient trajectories that end at diagonal vertices. 

Type (a) curves are provided by the following:

\begin{theorem}   \label{thm: holomorphic strip}
Given $\epsilon>0$ small, there exists a $J_{g(\epsilon s)}$-holomorphic strip $v$ from $\epsilon df_0$ to $\epsilon df_1$ that approximates $\gamma$. Moreover, $v$ is virtually extendable with regular virtual extension.
\end{theorem}

\begin{proof}
    This is just a rephrasing of the bijection, due to Floer~\cite{floer1988morse}, of complete gradient trajectories between Morse critical points and Floer trajectories, i.e., holomorphic strips with boundary on $\op{Graph}(\epsilon df_0)$ and $\op{Graph}(\epsilon df_1)$ between the corresponding intersection points, followed by stretching.
\end{proof}

Type (b) curves were constructed by Ekholm \cite{ekholm2007morse}, and they are also extendable with regular virtual extension. The cusp edges which the gradient trajectories on $\op{Sym}^\kappa C$ end on are of the form 
$$(c\times C^{\kappa-1})\setminus \Delta\subset \op{Sym}^\kappa C,$$ 
where $c$ is a cusp edge in $C$ (i.e., cusp edges in $\Sigma^\epsilon_\nu$ projected down to $C$) and $\Delta$ is the big diagonal.  Such cusp edges are of Morse-Bott type, but can be converted to Morse type. 

A Type (c) curve is simply a holomorphic strip $v:[A,B]_s\times[0,1]_t\to T^*C$ from $\epsilon df_0$ to $\epsilon df_1$ that approximates a trajectory $\gamma$ --- given by Theorem~\ref{thm: holomorphic strip} --- together with a branched double cover corresponding to the automorphism $(s,t)\mapsto (A+B-s,1-t)$ of the domain.  By Condition~\ref{condition: real exact}, $\gamma$ is not a subset of a closed trajectory and is hence somewhere injective.  Hence $v$, viewed as an orbifold holomorphic map to $\op{Sym}^2 C$ that passes through the diagonal, admits a virtual extension which is regular. (Let $\widetilde u: \widetilde \Sigma \to V$ be the lift of an orbifold holomorphic map $u:\Sigma\to V/G$ to a global quotient orbifold.  Then $u$ is regular if the linearized $\overline\bdry$-operator $L$ on $\widetilde u$ restricted to the $G$-invariant parts of the domain and codomain is surjective. This in turn is implied by the surjectivity of $L$ itself by an averaging procedure.)

\s\n
{\em $Y$-shaped connectors.} Next we describe the local models, called {\em $Y$-shaped connectors}, that correspond to neighborhoods of the trivalent vertices.  Let $w_{a,b,c,d}$ with $a<b<c$ be a holomorphic map $D_3\to \C_{z=x+iy}$ such that:
\be
\item $\bdry_0 D_3$, $\bdry_1 D_3$, $\bdry_2 D_3$ are mapped to $y=a$, $y=c$, $y=b$ and
\item $w_{a,b,c,d}$ restricts to a biholomorphism $\op{int}(D_3)\stackrel\sim\to \{ a< y< c\}\setminus \{y=b, x\leq d\}$.
\ee
Writing ${\bf a}=(a_1,\dots, a_n), {\bf b}= (b_1,\dots, b_n), {\bf c}=(c_1,\dots, c_n), {\bf d}=(d_1,\dots, d_n)$ with $a_i< b_i< c_i$ for all $i$, we define 
\begin{gather*}
    w_{{\bf a},{\bf b}, {\bf c},{\bf d}}=(w_{a_1,b_1,c_1,d_1},\dots, w_{a_n,b_n,c_n,d_n}): D_3\to \C^n,
\end{gather*}

\begin{figure}[ht]
    \centering
    \includegraphics[width=6cm]{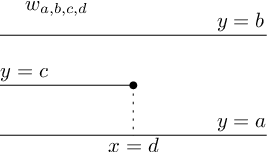}
    \caption{A component of a $Y$-shaped connector}
    \label{fig: picture 3}
\end{figure}

Let $f_1,f_2,f_3:\R^n\to \R$ be smooth functions such that $df_1$, $df_2$, and $df_3$ are constant and $\nabla(f_3-f_2)$ and $\nabla(f_2-f_1)$ are linearly independent. We may assume that:
\be
\item[(F)] All the coordinates of $\nabla(f_3-f_1)$, $\nabla(f_3-f_2)$, $\nabla(f_2-f_1)$ are positive, after a linear change of coordinates. 
\ee
We then set ${\bf a}=\nabla f_1$, ${\bf b}=\nabla f_2$, ${\bf c}=\nabla f_3$. This gives us a family of maps $w_{{\bf a},{\bf b}, {\bf c},{\bf d}}$, where ${\bf a}, {\bf b}, {\bf c}$ are fixed and ${\bf d}$ is variable. 

\begin{remark}
    There is a well-known (to the experts) error in the proof of gluing in Fukaya-Oh \cite{fukaya1996morse}.  There it is claimed that the $Y$-shaped connectors $w$ have surjective linearized $\overline\bdry$-operators 
    \begin{equation} \label{eqn: Dw}
        D_w: W^{1,p}_\delta(w^*T\C^n)\to L^p_\delta(\Lambda^{0,1}_{\C} w^*T\C^n),
    \end{equation}
    for $\delta>0$ small. The operators $D_w$, on the contrary, have $n$-dimensional cokernels, spanned by constant antiholomorphic differentials.  The Fukaya-Oh proof can be fixed in several more or less equivalent ways; see e.g., Ekholm \cite{ekholm2007morse} and Bao-Zhu~\cite{BaoZhu}.  The latter recasts the problem into an obstruction bundle gluing problem, together with a degree calculation. 
\end{remark}

The following gluing result is similar to that of \cite[Section 6]{ekholm2007morse}:

\begin{theorem} \label{thm: gluing along Y-shaped connectors}
   Given a folded Morse tree $(\pi_0:G\to T,\gamma:G\to C)$, one can glue the holomorphic strips of Types (a)--(c) into a holomorphic curve $v: \Theta_\lambda(T)\to \op{Sym}^\kappa (T^*C)$ corresponding to the Morse gradient graph $\gamma_\kappa: T\to \op{Sym}^\kappa C$, using $Y$-shaped connectors. The holomorphic curve $v$ is virtually extendable and regular.  There is a surjective degree $1$ map from the moduli space of virtually extendable holomorphic curves that approximate folded Morse trees and the moduli space of folded Morse trees. 
\end{theorem}

Here we are viewing the folded Morse tree as a gradient graph $\gamma_\kappa: T\to \op{Sym}^\kappa C$.  Note that in the definition of a folded Morse tree the diagonal vertices correspond to intersections with topmost stratum of the big diagonal of $\op{Sym}^\kappa C$.

The sketch of the proof of Theorem~\ref{thm: gluing along Y-shaped connectors} below is similar in spirit to \cite{BaoZhu}, and we thank Erkao Bao for explaining the proof to us.

\begin{proof}[Sketch of proof] $\mbox{}$

\s\n
{\em Part A: Basic gluing at a trivalent vertex.} In this part we sketch the most basic gluing at a trivalent vertex, leaving out estimates and not treating the most general case. 

Let $\epsilon>0$ be small. We make the following simplifications:
    \be
    \item the ambient manifold is a $2$-manifold $M$ and $f_1,f_2,f_3:M\to \R$ are generic smooth functions;
    \item the incoming trajectories $\gamma_1$ and $\gamma_2$ of $\nabla(f_3-f_2)$ and $\nabla(f_2-f_1)$ are rigid (modulo domain $\R$-translation) and emanate from critical points $p_1, p_2\in M$; 
    \item the outgoing trajectories $\gamma_3$ of $\nabla(f_3-f_1)$ come in a $1$-dimensional family (modulo domain $\R$-translation), given by translations in the direction normal to $\gamma_3$, and go to a critical point $q\in M$;
    \item $\gamma_1$ and $\gamma_2$ meet at $v$ and one of the $\gamma_3$ leaves from $v$;
    \item there is a local chart $U$ about $v=0$ --- chosen independently of $\epsilon$ --- such that $df_1,df_2,df_3$ are constant on $U$ and satisfy (F) on $U$, after a small perturbation on $U$.
    \ee
    The general case will require keeping track of the intersections of the unstable manifolds emanating from $p_1,p_2$ and the stable manifold emanating from $q$.
    
    Let $v_1,v_2, v_3$ be holomorphic strips corresponding to $\gamma_1,\gamma_2,\gamma_3$, with virtual extensions, furnished by Theorem~\ref{thm: holomorphic strip}, and let $w_0$ be a $Y$-shaped connector that agrees with $v_i$, $i=1,2,3$, up to the affine linear term.  

    For $i=1,2$, the holomorphic map $v_i:(-\infty,0]_s\times [0,1]_t \to \C^2$ is a half-strip, and there are two options to make the linearized $\overline\bdry$-operator $D_{v_i}$ Fredholm: (i) impose a boundary condition along $s=0$ (as is done in \cite{fukaya1997zero} and \cite{BaoZhu}) or (ii) {\em virtually extend} $v_i$ to a holomorphic map $\widetilde{v}_i$ with domain $(-\infty,\infty)\times[0,1]$. The portion $s>0$ is understood to be ``virtual'' (i.e., does not actually exist in our gluing problem) and is only used to invert the error term using $D_{\widetilde{v}_i}$; after inverting the error the result is restricted to the portion $s\leq 0$. Let $\widetilde{v}_i^*\widetilde{TC}$ be the ``pullback bundle''. One can verify that the operator 
    $$D_{\widetilde{v}_i}: W^{1,p}_\delta(\widetilde{v}_i^*\widetilde{TC})\to L^p_\delta (\Lambda^{0,1}_{\C}\widetilde{v}_i^*\widetilde{TC})$$
    has Fredholm index $\op{ind}=0$ and is an isomorphism.
    The case of $\widetilde{v}_3$ is analogous. 

    For each $Y$-shaped connector $w$, we modify the operator $D_w$ to $\widetilde D_w$ so that the domain is 
    \begin{equation} \label{eqn: enlarged domain}
        W^{1,p}_\delta(w^*T\C^2)\oplus \R\langle \rho_1, \rho_2 \rangle,
    \end{equation}
    where $\rho_1,\rho_2$ are defined as follows: Let $[0,\infty)_r \times [0,1]_t$ be the strip-like end corresponding to the positive/outgoing end of $D_3$. Let $\rho(r):[0,\infty)\to \R_{\geq 0}$ be a smooth step function such that $\rho(r)=0$ for $r\in[0,1]$ and $\rho(r)=1$ for $r\in [2,\infty)$. Then define $\rho_i$, $i=1,2$, to be equal to zero on $D_3\setminus ( [0,\infty)\times [0,1])$ and equal to $\rho(r)e_i$ on $[0,\infty)\times[0,1]$, where $e_i$ is a standard unit vector in $\R^2$. One can verify that the Fredholm index $\op{ind}(\widetilde D_w)=0$ and $\widetilde D_w$ is an isomorphism.

    We now apply pregluing to $(v_1,v_2,v_3,w_0)$ to obtain $u_{(0)}$. We use the usual Newton iteration to construct a genuine holomorphic curve, and explain the initial steps of the iteration.  The errors contributing to $\overline\bdry u_{(0)}$ that come from interpolating between $v_i$ and $w_0$ for $i=1,2$ can be inverted using $(\widetilde D_{w_0})^{-1}$, which, after quotienting out a term parallel to $\nabla (f_3-f_1)$, gives a term of the form 
    $$\xi+ c(\rho_1 -\rho_2), \quad \xi\in W^{1,p}_\delta(w^*T\C^2), \quad c\in \R,$$
    where $c(\rho_1-\rho_2)$ is the component in the normal direction. Next, the error $\overline\bdry c(\rho_1-\rho_2)$ cannot be inverted when the domain of $D_{v_3}$ is of class $W^{1,p}_\delta$; hence we have an obstruction bundle problem, and the obstruction section vanishes roughly when $v_3$ has been translated in its $1$-dimensional family so that $w_0+\xi+\overline\bdry c(\rho_1-\rho_2)$ agrees with $v_3$ up to first order.  The actual calculation reduces to a degree calculation, which is $1$.

\s\n
{\em Part B.}   In this part we explain how to apply the trivalent vertex gluing to the case at hand. Let $(\pi_0:G\to T,\gamma:G\to C)$ be a folded Morse tree, viewed as a Morse gradient graph $\gamma_\kappa: T\to \op{Sym}^\kappa C$.  We start with gradient trajectories of $\gamma_\kappa$ in $\op{Sym}^\kappa C$ that emanate from diagonal or cusp vertices.  The corresponding edge models with virtual extensions exist and are of Type (b) or (c). The gradient trajectories intersect at a point in $\op{Sym}^\kappa C$ away from the big diagonal, and the gluing to a Type (a) edge model can be done in $(T^*C)^\kappa$.  We inductively attach edges to obtain a virtually extendable holomorphic curve $v: \Theta_\lambda(T)\to \op{Sym}^\kappa (T^*C)$ corresponding to $\gamma_\kappa: T\to \op{Sym}^\kappa C$.
\end{proof}

\subsubsection{Step 4: Folded squid leg gluing} 

In this step we attach squid legs to the squid body from Step 2. This yields a holomorphic map $\mathcal{T}^{l}_{\nu, \epsilon}(\mathbf{p}_\alpha,\mathbf{y},\mathbf{p}_\beta)$ that is close to breaking into $(u,\gamma)\in \mathcal{P}^\nu_{k,l-k}({\bf q}, {\bf y}, {\bf q})$ for $\epsilon>0$ small and $\nu\gg 0$. 

This squid leg gluing is analogous to the gluing problem treated in Oh-Zhu \cite[Theorem 1.8]{OhZhu2024adiabatic}. In that paper the authors gave a new direct proof that the moduli spaces of pearly trajectories of Cornea-Lalonde~\cite{CorneaLalonde2006cluster} and Biran-Cornea~\cite{BiranCornea2009} are equivalent to the usual moduli space of holomorphic curves with Lagrangian boundary via adiabatic gluing.  (In \cite{BiranCornea2009} the equivalence is proved using a type of continuation argument which glues pearls together but not with gradient trajectories.)

Let $\epsilon>0$ be small and $\nu\gg 0$. Let $u^\circ$ be the holomorphic map close to $u$ that is furnished by Lemma~\ref{lemma: existence of ucirc} and let $v$ be the  virtually extendable holomorphic curve corresponding to the folded Morse tree $\gamma$ from Step 3.

\s\n
{\em Connectors.} We describe the local models, called {\em connectors}, that correspond to neighborhoods of the junctures between $u$ and $\gamma$.  Let $w_0$ be a holomorphic map 
$$D_2=\R_s\times[0,1]_t\to \C_{z=x+iy}$$ 
such that:
\be
\item $\bdry_0 D_2=\{t=0\}$ and $\bdry_1 D_2=\{t=1\}$ are mapped to $y=0$ and $y=1$ and
\item $w_0$ restricts to a biholomorphism $\op{int}(D_2)\stackrel\sim\to \{y>0\}\setminus \{y=1, x\leq 0\}$.
\ee
Note that $w_0$ is unique up to domain $s$-translation.  

The map $w_0$ can be modified by domain translation by $s=s_0$ and target translation by $x=x_0$.  Let $w_{c,d}$ be such a map, which is uniquely characterized by the following property: as $s\to \infty$, $w_{c,d}(s,t)$ approaches $c e^{\pi(s+it)}+ d$, where $c>0$ and $d\in \R$.
Writing ${\bf c}=(c_1,\dots,c_{2\kappa}), {\bf d}=(d_1,\dots, d_{2\kappa})$, we define 
\begin{gather*}
    w_{\bf c, \bf d}=(w_{c_1,d_1},\dots, w_{c_{2\kappa}, d_{2\kappa}}): D_2\to \C^{2\kappa}.
\end{gather*}

\begin{figure}[ht]
    \centering
    \includegraphics[width=6cm]{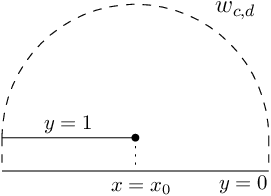}
    \caption{}
    \label{fig: picture 2}
\end{figure}

\s\n
{\em Gluing.} The gluing is very similar to the gluing from Step 3, and we explain the first steps of the Newton iteration. 

For simplicity we assume (SL1)--(SL4) below:
\be
\item[(SL1)] There is a single squid leg which is glued and the gluing occurs at $t=1/2$.
\ee
Viewing $u^\circ$ as a map $Y_1\to \op{Sym}^\kappa T^*C$, let $u^\circ_\tau$ be the restriction of $u^\circ$ to $\dot Y_{1,\tau}= Y_1\setminus \{\phi_{Y_1}(\tau)\}$, where $\tau\in (1/2-\lambda', 1/2+\lambda')$ and $\lambda'>0$ is small but independent of $\epsilon$. Choose a strip-like end
$$\mathcal{S}_\tau:=(-\infty,0]_s\times[0,1]_t\subset \dot Y_{1,\tau}$$ 
on a neighborhood of $\phi_{Y_1}(\tau)$, which depends on $\tau$ but not on $\epsilon$.
\be
\item[(SL2)] there exist local coordinates $U\subset \R^{2\kappa}$ for $\op{Sym}^\kappa C$ and $T^*U\subset T^*\R^{2\kappa}=\C^{2\kappa}$ for $\op{Sym}^\kappa T^*C$ such that the endpoint of $\gamma$ is $(0,0)\in T^*C$;

\item[(SL3)] $\gamma$ restricted to $U$ is a trajectory of $\nabla(f_1-f_0)$, where $\nabla f_0=0$ and $\nabla f_1=(1,\dots,1)$.
\ee
For $\epsilon>0$ small, rescaling $\C^{2\kappa}$ by $1/\epsilon$ sends $\op{Graph} \epsilon df_0$ and $\op{Graph} \epsilon df_1$ to $\op{Graph} df_0$ and $\op{Graph} df_1$.  
\be
\item[(SL4)] There exist $R_1(\epsilon), R_2(\epsilon)\gg 0$ with $R_2(\epsilon)-R_1(\epsilon)\gg 0$ such that the restriction of $\tfrac{1}{\epsilon} u^\circ_\tau$ to $\{-R_2(\epsilon)<s< -R_1(\epsilon)\}$ has boundary on $\op{Graph} df_0$ and $\op{Graph} df_1$ and limits to $w_{\bf c(\tau), \bf d(\tau)}$ for some ${\bf c}(\tau)$ and ${\bf d}(\tau)$, modulo domain $\R$-translation of $w_{\bf c(\tau), \bf d(\tau)}$. 
\ee
We note that $(\bf c(\tau), \bf d(\tau))$ is roughly linear in $\tau-1/2$ when $\epsilon>0$.

Now suppose $\kappa=2$ for simplicity. In this case we may take $v$ to be rigid. (Also note that $u^\circ$ is always rigid.) For each $u^\circ_\tau$ there is a single connector $w_{\bf c(\tau),\bf d(\tau)}$ as in (SL4) that agrees with the end of $\tfrac{1}{\epsilon} u^\circ_\tau$ ``up to first order'', and we may adjust $\tau$ so that $w_{\bf c(\tau),\bf d(\tau)}$ agrees with $v$ up to first order.  The signed count of such $\tau$ is $1$.

Choosing one such $\tau$, we preglue the triple $(v,w=w_{\bf c(\tau),\bf d(\tau)}, u^\circ_{\tau})$ to $u_{(0)}$ and invert the error $\overline\bdry u_{(0)}$ as before.

The linearized $\overline\bdry$-operator  
$$D_w: W^{1,p}_\delta(w^*T\C^2) \to L^p_\delta(\Lambda^{0,1}_{\C} w^*T\C^2)$$
can be analyzed in the same way as in \eqref{eqn: Dw}. It has $2$-dimensional cokernel and, following \eqref{eqn: enlarged domain} and its surrounding paragraph, we enlarge the domain and obtain
$$\widetilde D_w: W^{1,p}_\delta(w^*T\C^2)\oplus \R\langle \rho_1, \rho_2 \rangle \to L^p_\delta(\Lambda^{0,1}_{\C} w^*T\C^2),$$
which is an isomorphism.

As before, $(\widetilde D_{w})^{-1}(\overline\bdry u_{(0)})$, which, after quotienting out a term corresponding to $s$-translation along the strip-like end of $u^\circ_\tau$, gives a term of the form 
$$\xi+ c(\rho_1 -\rho_2), \quad \xi\in W^{1,p}_\delta(w^*T\C^2), \quad c\in \R,$$
where $c(\rho_1-\rho_2)$ is the component in the normal direction.  We then adjust $u^\circ_{\tau}$ so its negative end agrees with $w+c(\rho_1+\rho_2)$.

\s
Putting all the steps together, we have:

\begin{theorem}
    \label{thm: Morse-Floer-gluing}
    Let $l_0\geq 0$, $\nu>0$ sufficiently large, and $\lambda>0$ sufficiently small.  For $\epsilon>0$ sufficiently small, there is a surjective degree $1$ map from $\sqcup_{\alpha,\beta\in\mathfrak{J}}\mathcal{T}^{l}_{\nu,\epsilon}(\mathbf{p}_\alpha,\mathbf{y},\mathbf{p}_\beta)$ to $\mathcal{P}_{\nu}^{l}(\mathbf{q},\mathbf{y},\mathbf{q})$ for all $0\leq l\leq l_0$.
\end{theorem}

By Theorem~\ref{thm: Morse-Floer-gluing}, $\mathcal{F}^{\nu,\epsilon}_{\mathrm{Mor}}=\mathcal{F}^\nu_{\mathrm{hol}}$ mod $\hbar^{l+1}$ for $\nu\gg 0$.
Therefore, by Lemmas \ref{lemma-limit-F} and \ref{lemma-morse-compare} and taking the limit $l\to\infty$, we conclude the proof sketch of Conjecture~\ref{conj-morse}.

\printbibliography

\end{document}